\documentclass[11pt, twoside]{article}
\usepackage{amsmath,amsthm,amssymb}
\usepackage{enumerate, enumitem}
\usepackage{graphicx}
\usepackage{color}
\usepackage[colorlinks=true]{hyperref}
\def\titlerunning#1{\gdef\titrun{#1}}

\makeatletter

\def\author#1{\gdef\autrun{\def\and{\unskip, }#1}\gdef\@author{#1}}

\makeatother

\def\email#1{e-mail: #1}

\def\subjclass#1{{\renewcommand{\thefootnote}{}%
\footnote{\emph{Mathematics Subject Classification (2020):} #1}}}

\def\keywords#1{\par\medskip
\noindent\textbf{Keywords.} #1}
\newtheorem{theorem}{Theorem}[section]
\newtheorem{lemma}[theorem]{Lemma}
\newtheorem{proposition}[theorem]{Proposition}
\newtheorem{corollary}[theorem]{Corollary}

\hypersetup{linkcolor=blue,urlcolor=red,citecolor=red}

\theoremstyle{definition}

\newtheorem{remark}[theorem]{Remark}
\newtheorem{example}[theorem]{Example}

\numberwithin{equation}{section}

\begin{document}

\baselineskip=17pt

\titlerunning{}

\title{Observability  for   heat equations with time-dependent analytic memory}

%\author[author1]{Gengsheng Wang}
%\author[author2]{Yubiao Zhang}
%\author[author3]{Enrique Zuazua}

% /thanks = /footnotemark[num] + /footnotetext[num]{text}
\author{
Gengsheng Wang\footnotemark[2]
 \and
Yubiao Zhang\footnotemark[2] \footnotemark[3] \footnotemark[1]
 \and
Enrique Zuazua \footnotemark[3] \footnotemark[4] \footnotemark[5]
}

\renewcommand{\thefootnote}{\fnsymbol{footnote}}
\footnotetext[1]{
    Corresponding author. 
    \email{\texttt{yubiao\b{ }zhang@yeah.net}}
}
\footnotetext[2]{
    Center for Applied Mathematics, Tianjin University,
    	300072 Tianjin,  China.
        \email{\texttt{wanggs62@yeah.net}}
        }
\footnotetext[3]{
    Chair for Dynamics, Control, Machine Learning  and Numerics, Alexander von Humboldt-Professorship, Department of Mathematics,  Friedrich-Alexander-Universit\"{a}t Erlangen-N\"{u}rnberg,
    	91058 Erlangen, Germany. 
        \email{\texttt{enrique.zuazua@fau.de}}
        }
\footnotetext[4]{
    Chair of Computational Mathematics, Fundaci\'{o}n Deusto,
    	Av. de las Universidades, 24,
    	48007 Bilbao, Basque Country, Spain}
\footnotetext[5]{
    Departamento de Matem\'{a}ticas,
    	Universidad Aut\'{o}noma de Madrid,
    	28049 Madrid, Spain}

%\author{Gengsheng Wang\thanks{Center for Applied Mathematics, Tianjin University,
%		Tianjin, 300072, China. \email{wanggs62@yeah.net}}
%\and
%Yubiao Zhang\thanks{Center for Applied Mathematics, Tianjin University, Tianjin, 300072, China; \email{yubiao\b{ }zhang@yeah.net}}
%\and
%Enrique Zuazua\thanks{
%	[1] Chair for Dynamics, Control, Machine Learning  and Numerics, Alexander von Humboldt-Professorship, Department of Mathematics,  Friedrich-Alexander-Universit\"at Erlangen-N\"urnberg,
%	91058 Erlangen, Germany (\email{\texttt{enrique.zuazua@fau.de}}),
%	\newline \indent \indent
%	[2] Chair of Computational Mathematics, Fundaci\'{o}n Deusto,
%	Av. de las Universidades, 24,
%	48007 Bilbao, Basque Country, Spain,
%	\newline \indent \indent
%	[3] Departamento de Matem\'{a}ticas,
%	Universidad Aut\'{o}noma de Madrid,
%	28049 Madrid, Spain
%}
%}

\date{}

\maketitle


\begin{abstract}
This paper presents a complete analysis of the observability property of heat equations with time-dependent real analytic memory kernels.
More precisely, we characterize the geometry  of the space-time measurable observation sets  ensuring sharp observability inequalities, which are relevant both for control and inverse problems purposes.

Despite the abundant literature on the observation of heat-like equations, existing methods do not apply to models involving memory terms.

 We present a new methodology and observation strategy, relying
 on the decomposition of the flow, the time-analyticity  of solutions and the propagation of singularities.
This allows us to obtain a sufficient and necessary geometric condition on the measurable observation sets for sharp two-sided observability inequalities.  In addition,  some applications to control and relevant open problems are presented. 
\end{abstract}


%\tableofcontents

\section{Introduction}

\subsection{Equation and aim}
 Let  $\Omega\subset\mathbb R^n$ ($n\in \mathbb N^+:=\{1,2,3,\cdots\}$)  be a bounded domain with a $C^2$-boundary $\partial\Omega$ and consider  the   heat equation with  memory:
\begin{eqnarray}\label{our-obserble-system}
\left\{
\begin{array}{ll}
\partial_t y(t,x)  - \Delta y(t,x)
+ \displaystyle\int_0^t M(t-s) y(s,x)ds =0,
~&(t,x)\in \mathbb R^+ \times \Omega,\\
y(t,x)=0,~&(t,x)\in \mathbb R^+\times \partial\Omega,  \\
y(0,\cdot)=y_0(\cdot) \in L^2(\Omega),
\end{array}
\right.
\end{eqnarray}
where
 $\mathbb R^+:=(0,+\infty)$ and $M$ satisfies the following assumption:
\begin{itemize}[leftmargin=4em]
  \item[($\mathfrak H$)] 
 The time-dependent memory kernel $M$ is a non-trivial real analytic  function over $[0,+\infty)$.
\end{itemize}
Equation \eqref{our-obserble-system} has a unique solution,
denoted by  $y(\cdot,\cdot;y_0)$,
 in $C([0,+\infty);L^2(\Omega))$.

 Equations with memory have been the object of intensive study since the early works of  Maxwell  \cite{Maxwell}, Boltzmann \cite{Boltzmann-1,Boltzmann-2}, and Volterra \cite{Volterra-1,Volterra-2}
and the literature on the topic is broad (e.g., \cite{Amendola-Fabriio-Golden,Cattaneo, Christensen,Coleman-Gurtin,
Dafermos,Farbizio-Gogi-Pata,Gurtin-Pipkin, OP, Pandolfi-book, WZZ-1,  Yong-Zhang}
and the references therein).

The observability problem  we analyze in this paper is relevant in the context of inverse problems and control. Roughly it aims to recover the full information on solutions out of partial measurements.

The observability problem has been intensively investigated for  the pure heat equation, corresponding to the null memory kernel (i.e.,  \eqref{our-obserble-system} with $M=0$), the main goal being to
  recover full information about the solution at the final time $T>0$
 out of   its
  restriction to
  a space-time open subdomain or measurable set
   (see, for instance,
 \cite{Luis-wang, Fernandez-Cara-Zuazua, Fursikov-Imanuvilov, Lebeau-Zuazua, Lin, Phung-Wang, Zuazua}).

  Compared with the pure heat equation, the study on the observability for
    heat equations with memory
   is lagging behind, despite some
 key works    (\cite{Barbu-Iannelli, Chaves-Silva-Zhang-Zuazua,  Fu-Yong-Zhang,
    Guerrero-Imanuvilov, Halanay-Pandolfi,    Pandolfi-book, Zhou-Gao}
     and the references therein).
  In \cite{Zhou-Gao} (see also \cite{Guerrero-Imanuvilov, Halanay-Pandolfi} for the case of boundary observation and control), the
authors found that when the observation set is time-invariant (of the form $(0,T)\times\omega$ for some open subset  $\omega\subset\Omega$) and proper,
  the null observability inequality
  for heat equations with memory does not hold.
    As a remedy, in  \cite{Chaves-Silva-Zhang-Zuazua}, the authors
    proposed a moving or    time-varying  observation mechanism to ensure
    the observability property. However,
              the necessity  of the aforementioned  geometric condition was not shown.
%>>>>>>> New version on 23-10-2023

Observability and controllability problems have also been considered for other PDEs involving memory terms. See, for example, \cite{Loreti-Sforza, Lu-Zhang-Zuazua}  for wave equations and \cite{Fernandez-Machado-Souza} for Stokes equations with memory, respectively. Here we focus on the heat equation with memory \eqref{our-obserble-system}.

%<<<<<<<

This paper is devoted to  proving a necessary and sufficient geometric condition on
 measurable subsets $Q\subset \mathbb R^+ \times \Omega$
 guaranteeing  the observability of  equation  \eqref{our-obserble-system}, so that  the restriction
 to $Q$ suffices to fully recover the solution.
   The  geometry of the observation sets
  and the specific form of the  observability inequalities studied in the current paper are motivated by  the following property $(\mathcal P)$ of the dynamics :
   \begin{itemize}[leftmargin=4em]
    \item[$(\mathcal{P})$]\label{property-P} The singularities of the solutions of equation \eqref{our-obserble-system} propagate
        along the $t$-axis as follows:
          \begin{itemize}
            \item[$\bullet$] For $t>0$,
       singularities propagate without any change on the regularity order.
          \item[$\bullet$] The backward propagation of singularities to the initial data results in the loss of $4$ space-derivatives.
    \end{itemize}
     \end{itemize}
 This   property   exhibits and reflects both, the singularity propagation properties  for $t>0$, proper to hyperbolic systems, but also
     a boundary layer at $t=0$, which does not occur  in the hyperbolic setting, being rather of a parabolic nature. Note however that for classical heat-like equations, the backward propagation of singularities to the initial data results in the loss of infinitely many space-derivatives, while for our model the loss is exactly $4$, and this  is true for all non-trivial analytic memory kernels.

     These two fundamental aspects of the dynamics under consideration,   constitute  a manifestation of the hybrid  heat-wave nature of equation \eqref{our-obserble-system}.
The numerical simulations in Section \ref{202310-yb-SimpleExplanationForDecomposition}  confirm and provide computational evidence of this hybrid character of the dynamics.

  Property $(\mathcal{P})$ is induced by  the decomposition of the flow
              generated by equation \eqref{our-obserble-system}, recently developed
     in \cite[Theorem 1.1]{WZZ-1} (and recalled in Theorem \ref{cor-0423-demcomposition}  in the Appendix at the end of this paper for the sake of completeness). It states that
               the flow
   \begin{align}\label{varPhi-y-y0}
\varPhi(t)y_0:=y(t,\cdot;y_0) \in L^2(\Omega),~ y_0 \in L^2(\Omega),~t\geq 0
\end{align}
 can be decomposed into  three components: a heat-like one $\mathcal P_N(t)$, a wave-like one $\mathcal W_N(t)$, and a  remainder  $\mathfrak R_N(t)$.

\subsection{Main results}  
We start by introducing the following  definitions:
  \begin{itemize}[leftmargin=4em]
\item[(D1)]  (\textit{Space $\mathcal H^s$}) The Laplacian is denoted by
 \begin{eqnarray}\label{selfadjoit-elliptic-operator}
A f := \Delta f,
\;\;\mbox{with its domain}\;\;
D(A):=H^2(\Omega)\cap H_0^1(\Omega).
\end{eqnarray}
 Let $\eta_j$  be the $j^{th}$ eigenvalue of $-A$ and $e_j$
be the corresponding normalized eigenfunction in $L^2(\Omega)$.
  For each $s\in \mathbb{R}$, we define the real Hilbert space
\begin{eqnarray}\label{def-space-with-boundary-condition}
\mathcal H^s :=
\bigg\{f= \sum_{j=1}^\infty a_j e_j ~:~ (a_j)_{j\geq 1}\subset \mathbb R,~
\sum_{j=1}^\infty |a_j|^2 \eta_j^{s}<+\infty
\bigg\},
\end{eqnarray}
equipped with the inner product
\begin{eqnarray}\label{def-Hs-space-norm}
\langle f, g\rangle_{\mathcal H^s}   :=
\sum_{j=1}^\infty a_j b_j \eta_j^{s},\;
\;\;\;\;
f= \sum_{j=1}^\infty a_j e_j \in \mathcal H^s
\;\;\mbox{and}\;\;
g= \sum_{j=1}^\infty b_j e_j \in \mathcal H^s,
\end{eqnarray}
  where  $(a_j)_{j\geq 1}, (b_j)_{j\geq 1}\subset \mathbb R$.

  \item[(D2)]   (\textit{Moving observation  condition}, MOC for short)
The triplet $(Q,S,T)$, where  $T>S\geq 0$ and $Q$ is a nonempty measurable subset of
 $\mathbb R^+ \times \Omega$,
   is said to satisfy the MOC if
  \begin{align}\label{Q-S-T-bounded-from-below}
    \mathcal T_{\Omega}(Q,S,T)
    :=\underset{x\in\Omega}{\text{ess-inf}} \int_S^T \chi_Q(t,x) dt
    >0
  \end{align}
  (here and below, for each set $E$,  $\chi_E$ denotes its  characteristic function).

 \end{itemize}

  \begin{remark}\label{remark1.1,2-22}
    In the definition (D2), the following points should be noted:
    \begin{itemize}[leftmargin=4em]
        \item[$(i)$]    The  MOC condition is  natural in view of  the wave-like nature of the flow  and, more precisely, is dictated by the propagation of singularities in the  $t$-direction as described in the aforementioned property $(\mathcal P)$.
       Indeed, for each $x\in \Omega$, the slice
        $Q_x :=\{(t,x)\in Q : t\in [S,T]\}$ lies on the half-line
        $\{(t,x) : t\geq 0\}$, which is
        the characteristic line  issued from $x$ at time $t=0$ (see
        \cite[$(iv)$ and $(ii)$ of Theorem 1.2]{WZZ-1}, where  the propagation of singularities was introduced).
        We write
        \begin{eqnarray*}%\label{1.6-3-18}
            \mathcal{T}_{x,S,T}:=|Q_x|=\int_S^T \chi_Q(t,x) dt,
        \end{eqnarray*}
        where
        $|Q_x|$ denotes the measure of $Q_x$ in  $\mathbb{R}$.
        Then, $\mathcal{T}_{x,S,T}$
        can be viewed  as
        the cumulative time that the characteristic line spends in the
        observation set $Q$
        during the time period $[S,T]$. Consequently, \eqref{Q-S-T-bounded-from-below}
        means that the  cumulative time $\mathcal{T}_{x,S,T}$  has  a uniform positive lower bound
        with respect to a.e. $x\in\Omega$. This guarantees that all propagating singularities of the solutions (to equation \eqref{our-obserble-system}) spend an uniform percentage of time in the observation set $Q$.

\item[$(ii)$]
The  MOC was introduced in \cite[Assumption 4.1]{Chaves-Silva-Zhang-Zuazua}
    (see also \cite{Chaves-Silva-Rosier-Zuazua}).  But in these works some extra technical assumptions on $Q$ were imposed for the employment of Carleman inequalities.
       In the most general definition of the MOC considered here, $Q$ can be a measurable set, while in the definition of
\cite[Assumption 4.1]{Chaves-Silva-Zhang-Zuazua},
 $Q$ was an open set.
Our MOC
 is strictly weaker than
\cite[Assumption 4.1]{Chaves-Silva-Zhang-Zuazua} (see Example \ref{example-mcc}  for a concrete example of the MOC), because of the absence of any other geometric technical conditions, that in \cite{Chaves-Silva-Zhang-Zuazua} were motivated by the use of Carleman inequalities.

 \end{itemize}

    \end{remark}

\vskip 5pt
  The main result of this paper is as follows.
  \begin{theorem}\label{202205TJyb-MainTheorem-merged}
  Let $M$ be a nonzero real analytic function over $[0,+\infty)$.
%%>>>>>>> Old version 10-11-2022
%Let $M$ be a n analytic kernel as in $\mathfrak H$.
%%<<<<<<<
   Let $T>S> 0$ and $Q$ be a nonempty measurable subset of
  $(0,+\infty) \times \Omega$.
  The following two statements are equivalent:
  \begin{itemize}
      \item [(i)] The triplet $(Q,S,T)$ satisfies the MOC.

      \item [(ii)]  There is a constant $C>0$ such that
      \begin{align}\label{202205TJYB-TwoSideObservability}
          \frac{1}{C} \| y_0\|_{ \mathcal H^{-4} }
          \leq  \int_S^T  \big\| \chi_Q(t,\cdot) y(t,\cdot;y_0) \big\|_{L^2(\Omega)}   dt
          \leq  C \| y_0\|_{ \mathcal H^{-4} }
      \end{align}
      for every solution of equation \eqref{our-obserble-system} with $y_0 \in L^2(\Omega)$.
  \end{itemize}

  \end{theorem}

\begin{remark}%\label{remark14.-4-8}
The following comments are worthy of consideration:
\begin{itemize}[leftmargin=4em]
    \item[(a1)]  Theorem \ref{202205TJyb-MainTheorem-merged} guarantees that
     when $S>0$,
     the MOC satisfied by $(Q,S,T)$ is a sufficient and necessary condition to ensure
     the  two-sided observability inequality  \eqref{202205TJYB-TwoSideObservability}.
%     Because the MOC is led by the property $(\mathcal P)$ (see $(i)$ of Remark \ref{remark1.1,2-22}), from the equivalence in Theorem \ref{202205TJyb-MainTheorem-merged}, the  two-sided observability inequality \eqref{202205TJYB-TwoSideObservability} is also led by  the property $(\mathcal P)$.

               ~~The first inequality in \eqref{202205TJYB-TwoSideObservability}  guarantees that, by measuring a solution of equation \eqref{our-obserble-system}
    over $Q$ in the space $L^1(S,T;L^2(\Omega))$, one can recover its initial datum in  $ \mathcal H^{-4}$, while the second inequality in \eqref{202205TJYB-TwoSideObservability} shows its optimality.

 \item[(a2)]  By density, the inequalities in \eqref{202205TJYB-TwoSideObservability} hold for all $y_0\in \mathcal H^{-4}$.  In fact, the regularity estimate in \cite[Theorem 1.4]{WZZ-1} states that when $T>S>0$, there is a $C>0$ such that
 \begin{align*}%\label{20221110-tjYB-regularity}
     \| y(\cdot,\cdot;y_0)  \|_{C([S,T];L^2(\Omega))}
     \leq  C  \| y_0  \|_{\mathcal H^{-4}},
     ~\forall\, y_0  \in \mathcal H^{-4}.
 \end{align*}
%>>>>>>> New version on 10-11-2023
Moreover, the fact that the norm $\mathcal H^{-4}$ of the initial datum is the one observed is natural in view of the discussion above and the emergence of singularities of order $4$ when solving the dynamics backwards in time at $t=0$. In fact, away from $t=0$ the dynamics can be decomposed as
\begin{equation}\label{202310-yb-SimpliedKeyEquality}
    y(t,\cdot;y_0) = - M(t) A^{-2} y_0 +
        \text{``small terms"},~~0<S\leq t \leq T
\end{equation}
(see \eqref{202310-yb-RewriteSimpliedKeyEquality}  in the later section, for a rigorous analysis).
    \item[(a3)]     To ensure  Theorem \ref{202205TJyb-MainTheorem-merged},
    it is necessary that $S>0$ (see
      Theorem \ref{202210-thm-weighted-observability}
for the case $S=0$). The treatment of the more delicate case $S=0$
     requires involving the weight function $t^\alpha$ (with $\alpha>1$) into the integrand
      in \eqref{202205TJYB-TwoSideObservability}  (see also
      Theorem \ref{202210-thm-weighted-observability}).

  \item[(a4)]
The proof of  Theorem \ref{202205TJyb-MainTheorem-merged} employs the wave-like aspects of the dynamics \eqref{our-obserble-system}. Inspired by the observability for wave equations (see for instance \cite{BLR92,RauchTaylor}), we develop the following modified three-step strategy to prove Theorem \ref{202205TJyb-MainTheorem-merged}:
\begin{itemize}[leftmargin=4em]
    \item[Step 1.] We establish a relaxed observability inequality (see Lemma \ref{lem-relaxed-exact-ob}), based on the simplified equality \eqref{202310-yb-SimpliedKeyEquality}, which is a consequence of the decomposition in Theorem \ref{cor-0423-demcomposition},  and especially of the wave-like component.
    \item[Step 2.] We obtain
    a qualitative unique continuation property for solutions of  equation \eqref{our-obserble-system} (see Lemma \ref{lem-unique-continuation-for-low-frequence}). Our proof  relies on the time-analyticity property  (stated in Proposition \ref{pro-analyticity-pointwise}) of the solutions of equation \eqref{our-obserble-system}.
    \item[Step 3.] We conclude the exact observability inequality \eqref{exact-ob-twosides-by-yb-202110},  by means of the classical compactness-uniqueness argument.
\end{itemize}

  \item[(a5)]   Theorem \ref{202205TJyb-MainTheorem-merged} remains true, as can be proven by similar arguments,  when the norm in $L^1_tL^2_x$ in \eqref{202205TJYB-TwoSideObservability} is replaced by that in $L^p_tL^2_x$ (with $1<p \leq +\infty$),
 i.e., when $T>S>0$,  the triplet  $(Q,S,T)$ satisfies the MOC if and only if
   there is a  $C>0$ such that
      \begin{align*}
          \frac{1}{C} \| y_0\|_{ \mathcal H^{-4} }
          \leq
          \| \chi_Q y(\cdot,\cdot;y_0) \|_{ L^p(S,T;L^2(\Omega)) }
          \leq  C \| y_0\|_{ \mathcal H^{-4} }
          ~~~\mbox{for all}~ y_0 \in L^2(\Omega).
      \end{align*}
  \item[(a6)] In Section \ref{sec-further-studies}, we present some further developments on Theorem \ref{202205TJyb-MainTheorem-merged}, and in Section \ref{sec-controllability}, we discuss some applications of Theorem \ref{202205TJyb-MainTheorem-merged}
  to  the control of system \eqref{our-obserble-system}.
\end{itemize}
\end{remark}

\vskip 5pt

\subsection{Discussion on the contributions}
The  main contributions  of the results of this paper are as follows:
  \begin{itemize}[leftmargin=4em]
    \item[(b1)]
        The two-sided observability inequality \eqref{202205TJYB-TwoSideObservability}
    and  the necessary and sufficient MOC condition on
        $(Q,S,T)$.

    \item[(b2)]    The optimality and minimality  of our MOC condition,  involving measurable (not necessarily open) observation sets.

    \item[(b3)] The methods  in this paper themselves are also new and can be of independent use to tackle other problems related with parabolic memory models and in particular inverse problems and long time asymptotics.  \end{itemize}

\subsection{Organization of the paper}
 The rest of this paper is organized as follows.
 In Section \ref{sec-preliminaries} we present some preliminary results.
%provides some preliminaries.
In Section \ref{sec-main-proofs},  Theorem \ref{202205TJyb-MainTheorem-merged} is proven.
Section \ref{sec-further-studies} presents some extensions of Theorem \ref{202205TJyb-MainTheorem-merged}.
Section \ref{sec-controllability} shows some applications in control.
Section \ref{202310-yb-SimpleExplanationForDecomposition} provides numerical simulations for the hybrid heat-wave nature of equation \eqref{our-obserble-system}.
Section \ref{sec-open-problems} is devoted to discussing several open problems.
Section \ref{sec-appendix} is devoted to the technical Appendix.

\section {Preliminaries}\label{sec-preliminaries}

\subsection{Properties of the flow}

% Recall that we have included
% \cite[Theorem 1.1]{WZZ-1} in the Appendix of this paper (see Theorem \ref{cor-0423-demcomposition}).
%It deserves mentioning that by
According to Theorem \ref{cor-0423-demcomposition} in the Appendix of this paper (see also \cite[Theorem 1.1]{WZZ-1}), we know
that
for each $t\geq 0$, the operator  $\varPhi(t)$ (given in \eqref{varPhi-y-y0}) constitutes
an element of the space $\mathcal{L}(\mathcal H^s)$ for any $s\in \mathbb{R}$.

In this subsection, we present some properties of the flow $\varPhi(t)$,  which are consequences of
Theorem \ref{cor-0423-demcomposition}. %, and that will be used in the proofs of our main results.

 \begin{proposition}\label{cor-heat-wave-N=2-by-yb-20211126}
		There is an $\mathcal R_c \in
	C\big(\mathbb R^+; C(\mathbb R^+)  \big)$ such that
	\begin{eqnarray}\label{demcomposition-N=2-PW}
	\varPhi(t)= e^{tA}\Big(1- tM(0) A^{-1} + M(0) A^{-2} \Big)
-M(t)  A^{-2} + \mathcal R_c(t,-A) A^{-3},\;\;t>0,
	\end{eqnarray}
	and such that  for each $s\in \mathbb R$, $\mathcal R_c(\cdot,-A)$ belongs to $C(\mathbb R^+;\mathcal L( \mathcal H^s))$ and satisfies
	\begin{eqnarray}\label{R1R2-regularity-by-yb-20211126}
 	\big\| \mathcal R_c(t,-A) \|_{
		\mathcal L( \mathcal H^s) }
	\leq
	\exp\Big(  2(1+t) \big(1+\|M\|_{ C^{2}([0,t])} \big)  \Big),~~t>0.
	\end{eqnarray}
\end{proposition}

\begin{proof}
Apply Theorem \ref{cor-0423-demcomposition} (with $N=2$)  to get 
	\begin{align}\label{4.23,7.25}
	\varPhi(t)= e^{tA}  + e^{tA}\Big( - p_0(t) A^{-1} + p_1(t) A^{-2} \Big)
+  h_1(t) A^{-2} 	- R_2(t,-A) A^{-3},\;\;t>0.
	\end{align}
	At the same time, it follows from \eqref{thm-ODE-meomery-asymptotic-estimate-hypobolic} that for each $t>0$,
	\begin{align}\label{4.23,7.25,11:33}
	\begin{cases}
	p_0(t)= M(0 )t, 
	   \\
	p_1(t)= M(0) - M^\prime(0)  t + \frac{1}{2} M(0)^2 t^2,
        \\
	h_1(t)=-M(t).
	\end{cases}
	\end{align}
 From equations \eqref{4.23,7.25} and \eqref{4.23,7.25,11:33}, we find
	\begin{align}\label{del-cor4.6-1}
	\varPhi(t)=& e^{tA}\Big(1- M(0) tA^{-1} + M(0) A^{-2} \Big) -M(t) A^{-2}
	\nonumber\\
	&
+e^{tA} \Big(
	 - M^\prime(0)  t + \frac{1}{2} M(0)^2 t^2
	\Big) A^{-2} -  R_2(t,-A) A^{-3}
	,\; t>0.
	\end{align}

	Next, we define
\begin{align*}
\mathcal R_c(t,\tau) := e^{-t\tau} \Big(
	 - M^\prime(0)  t + \frac{1}{2} M(0)^2 t^2
	\Big) (-\tau) -  R_2(t,\tau),\;\;t>0,~ \tau>0.
\end{align*}
Then, by spectral functional calculus, we have
	\begin{align}\label{4.26,7.25}
	\mathcal R_c(t,-A) = e^{tA} \Big(
	 - M^\prime(0)  t + \frac{1}{2} M(0)^2 t^2
	\Big) A -  R_2(t,-A),\;\;t>0.
	\end{align}
	At the same time, we can directly check that  for each  $s\in \mathbb R$,
	\begin{align}\label{4.27,7.25}
\| A e^{tA} \|_{\mathcal L( \mathcal H^s )}
	\leq  t^{-1}
	,~~
	t>0.
		\end{align}
	Now, \eqref{demcomposition-N=2-PW} follows from \eqref{del-cor4.6-1}
	and \eqref{4.26,7.25}, while \eqref{R1R2-regularity-by-yb-20211126} follows from \eqref{4.26,7.25},
	 \eqref{4.27,7.25}, and \eqref{RN-property-regularity} (where $N=2$).
	This completes the proof.
\end{proof}

The proofs of the following two corollaries are presented in Subsection \ref{Appendix-SeveralProofs} in the Appendix.

\begin{corollary}\label{cor-only-h1}
	Let $\beta \in [2,3]$.  Then, there is an $\widehat{\mathcal R}_c \in
	C\big(\mathbb R^+; C(\mathbb R^+)  \big)$ such that
	\begin{eqnarray}\label{demcomposition-N=2-only-h1}
	\varPhi(t)=-M(t)  A^{-2} + \widehat{\mathcal R}_c(t,-A) (-A)^{-\beta},\;\;t>0.
	\end{eqnarray}
	Moreover,  for each $s\in \mathbb R$, $\widehat{\mathcal R}_c(\cdot,-A)$ belongs to $C(\mathbb R^+;\mathcal L( \mathcal H^s))$ and satisfies for some $C>0$,
	\begin{eqnarray}\label{R1R2-regularity-only-h1}
	\big\| \widehat{\mathcal R}_c(t,-A) \|_{
		\mathcal L( \mathcal H^s) }
	\leq   \frac{C}{ t^{\beta} }
	\exp\Big(  2(1+t) \big( 1+\|M\|_{ C^{2}([0,t]) } \big) \Big),~~t>0.
	\end{eqnarray}
\end{corollary}

\begin{corollary}\label{cor-only-parabolic}
	There is an $\widetilde{\mathcal{R}}_c\in C\big(\mathbb R^+; C(\mathbb R^+)  \big)$
	such that
	\begin{eqnarray}\label{demcomposition-N=2-only-parabolic}
	\varPhi(t) = e^{tA}  +  \widetilde{\mathcal{R}}_c (t,-A) A^{-2},~~
	t>0.
	\end{eqnarray}
	Moreover, for each $s\in \mathbb R$, $\widetilde{\mathcal{R}}_c (\cdot,-A)$ belongs to
	$C(\mathbb R^+;\mathcal L( \mathcal H^s))$ and satisfies for some $C_1>0$,
	\begin{eqnarray*}\label{R1R2-regularity-only-parabolic}
		\big\| \widetilde{\mathcal{R}}_c (t,-A) \|_{
			\mathcal L( \mathcal H^s) }
		\leq  C_1 \exp\Big( 2(1+t) \big(1+\|M\|_{ C^{2}([0,t])} \big) \Big)
,~~t>0.
	\end{eqnarray*}
	\end{corollary}

\subsection{Time-analyticity of solutions}
This subsection is concerned with the analyticity of the solutions to equation \eqref{our-obserble-system} with respect to the time variable.

\begin{proposition}\label{pro-analyticity-pointwise}
Let  $y_0 \in \cup_{s\in \mathbb R} \mathcal H^{s}$ be such that  
$(A^{-2}y_0)|_{\omega} \in L^2_{loc}(\omega)$ 
for some nonempty open subset $\omega\subset \Omega$. 
%the restriction of $A^{-2}y_0$ belongs to $L^2_{loc}(\omega)$, i.e.,
%$(A^{-2}y_0)|_{\omega} \in L^2_{loc}(\omega)$. 
Then, the following statements are true:
\begin{itemize}
  \item[] (i) The restriction of the solution $y(\cdot,\cdot;y_0)$ over $(0,+\infty) \times\omega$ belongs to $L^2_{loc}((0,+\infty) \times\omega)$.
  \item[] (ii) For a.e. $x\in \omega$, the function $t\mapsto y(t,x;y_0)$   is real analytic over $(0,+\infty)$.
\end{itemize}
\end{proposition}

\begin{remark}
We remark that in general, the solutions of \eqref{our-obserble-system} are  not analytic in the space variable because of the finite-order regularizing effect of $\varPhi(t)$ (see \cite[Theorem 1.4]{WZZ-1}).
This shows a difference between the heat semigroup $\{e^{tA}\}_{t\geq 0}$ and $\{\varPhi(t)\}_{t\geq 0}$
from the perspective of the analyticity of solutions.
\end{remark}

\begin{proof}[Proof of Proposition \ref{pro-analyticity-pointwise}]
 Since
 \begin{align}\label{202202yb-ThmAnalyticityRevised-PropertiesOfY0}
   (A^{-2}y_0)|_{\omega} \in L^2_{loc}(\omega)
    ~\text{and}~
    y_0\in \mathcal H^{-2m}
    ~\text{for some}~
    m\in \mathbb N^+,
 \end{align}
 it follows from  \eqref{0423-demcomposition-eq} and \eqref{def-PN-HN-RN} with $N=n+m$ ($n$ is the space dimension) in the Appendix  that for each $t>0$,
\begin{align}\label{solution-decomposition-N=d+2}
   y(t,\cdot;y_0) = \varPhi(t)y_0
   =& e^{tA} y_0  +  e^{tA} \sum_{l=0}^{n+m-1}  p_l(t) (-A)^{-l-1} y_0
   \nonumber\\
   & +  \sum_{l=1}^{n+m-1}  h_l(t) (-A)^{-l-1} y_0
   +  \mathfrak R_{n+m}(t) y_0,
\end{align}
where $p_l,h_l$, and $\mathfrak R_{n+m}$ are given by \eqref{thm-ODE-meomery-asymptotic-estimate-hypobolic} and 
\eqref{def-PN-HN-RN}-\eqref{0921-RN-good-remainder}, respectively.  By \eqref{solution-decomposition-N=d+2},
we see that to  prove conclusions $(i)$ and $(ii)$, it suffices to show the following three assertions:
\begin{itemize}[leftmargin=4em]
  \item[(A1)] For each $l\in \mathbb N$, both  $p_l$ and $h_l$ are real analytic over $(0,+\infty)$;
  \item[(A2)] For each $l\in \mathbb N^+$,  $(-A)^{-l-1}y_0\in\mathcal H^{-2m}$ and
  $[(-A)^{-l-1}y_0]|_\omega\in L^2_{loc}(\omega)$;
  \item[(A3)] For each $z\in \mathcal H^{-2m}$, the functions $t\mapsto e^{tA}z$ ($t>0$) and $t\mapsto\mathfrak R_{n+m}(t) z$ ($t>0$) belong  to $C(\mathbb R^+;L^2(\Omega))$. Moreover,
      for each $z\in \mathcal H^{-2m}$ and for a.e. $x\in \Omega$,
      the functions $t \mapsto (e^{tA}z)(x)$ ($t>0$) and $ t \mapsto (\mathfrak R_{n+m}(t) z)(x)$ ($t>0$)  are real analytic.
           \end{itemize}
           We now prove (A1)--(A3). First,  (A1) follows from \eqref{thm-ODE-meomery-asymptotic-estimate-hypobolic} and the analyticity of $M$ over $[0,+\infty)$. Second, (A2) can be checked directly from \eqref{202202yb-ThmAnalyticityRevised-PropertiesOfY0} and the iterative use of the following property:
 \begin{align*}
   h\in \mathcal H^{-2m} \cap L^2_{loc}(\omega) \Rightarrow
   A^{-1}h \in \mathcal H^{-2m} \cap L^2_{loc}(\omega)
 \end{align*}
(the above property is due to the ellipticity of the operator $A=\Delta$ and is a direct consequence of  \cite[Theorem 18.1.29]{Hormander-3}).

The remaining task is to show (A3).
To this end, we arbitrarily fix $z \in \mathcal H^{-2m}$. Then, we have
\begin{align}\label{y0-Galerkin-coefficients}
   z = \sum_{j\geq 1} a_j \eta_j^m e_j
   ~~\text{for some}~
   (a_j)_{j\geq 1} \in   \ell^2
\end{align}
(recall that   $\eta_j$  is the $j^{th}$ eigenvalue of $-A$ and $e_j$
is the corresponding normalized eigenfunction in $L^2(\Omega)$).
From \eqref{y0-Galerkin-coefficients},  \eqref{def-PN-HN-RN}, and \eqref{0921-RN-good-remainder} in the Appendix,  we find
\begin{align}\label{solution-Galerkin-with-fj-and-L1bound}
   e^{tA} z =  \sum_{j\geq 1} f_j(t) \eta_j^{-n} a_j e_j
   ~\text{and}~
   \mathfrak R_{n+m}(t) z = \sum_{j\geq 1} g_j(t) \eta_j^{-n} a_j e_j,~
   t>0,
\end{align}
where
\begin{align}\label{def-fj-expression-by-yb}
   f_j(t):= \eta_j^{n+m}e^{-\eta_j t}
   ~\text{and}~
   g_j(t):=  \int_0^t e^{-\eta_j s} \partial_s^{n+m} K_M(t,s)  ds,~t>0.
\end{align}
Next, we will examine  the analyticity of $f_j$ and $g_j$ ($j\in \mathbb N^+$).
We claim the following:

 {\it There is an open subset $\mathcal O$ in $\mathbb C$ with $\mathcal O \supset(0,+\infty)$ such that each $g_j$ (resp., $f_j$) can be extended to be  an analytic function over $\mathcal O$, denoted by $\tilde g_j$ (resp., $\tilde f_j$), with the following estimate:
\begin{align}\label{gj-fj-uniform-estimates}
   \sup_{j\geq 1} \| \tilde g_j \|_{C(G)} < +\infty
   ~\Big(\text{resp.,  $\sup_{j\geq 1} \| \tilde f_j \|_{C(G)} < +\infty$}\Big)   ~\text{for each}~
   G\Subset \mathcal O.
\end{align}}
For this purpose, we use the real analyticity of  $K_M$  over the set $S_+:=\{(t,s)\in \mathbb R^2 ~:~ t\geq s\}$ (see Proposition \ref{prop-varPhi-expression} in the Appendix) to
 obtain an open subset $\mathcal D \supset S_+$ in $\mathbb C^2$ such that $K_M$ can be extended
 to an analytic function over
$\mathcal D$. We still use $K_M$ to denote this extension. We define
\begin{align*}
   \mathcal O := \Big\{ t \in \mathbb C  ~:~
   \text{Re}\,t>0\;\;\mbox{and}\;\;
     (t, t\tau) \in \mathcal D \;\mbox{for each}\;\tau\in[0,1]\Big\}
      \supset (0,+\infty),
\end{align*}
and for each $j\in \mathbb N^+$, we define the following function over $\mathcal O$:
\begin{align}\label{2.21,2-25}
  \tilde g_j(t) := t\int_0^1 e^{-\eta_j \tau t} \partial_s^{n+m} K_M(t,\tau t)  d\tau,~t\in\mathcal O.
\end{align}
Several properties on $\tilde g_j$ are given. First,
it follows from \eqref{2.21,2-25} and the analyticity of  $K_M$
that
 $\tilde g_j$ is analytic over $\mathcal O$. Second, it follows from
 \eqref{2.21,2-25} and
\eqref{def-fj-expression-by-yb} that  $\tilde g_j|_{(0,+\infty)}=g_j$.
Third, it follows from \eqref{2.21,2-25}
that
\begin{align*}%\label{fj-extension-estimate-by-yb-202112}
   | \tilde g_j(t)|  \leq&  |t|
     \Big( \sup_{0< \tau < 1} |\partial_s^{n+m} K_M(t,\tau t)| \Big),~t\in \mathcal O.
\end{align*}
From the above properties, we see that the above claim is true for $g_j$.
Similarly, we can show that it also holds for $f_j$. Thus, we have proven the above claim.

 Finally, by \eqref{y0-Galerkin-coefficients}, we have
 \begin{align*}
   \int_{\Omega} \sum_{j\geq 1}   |a_j e_j(x)|^2   dx
   = \sum_{j\geq 1} |a_j|^2
    < +\infty,
 \end{align*}
 which shows that
\begin{align}\label{aj-ejx-l2-convergence}
 \sum_{j\geq 1} |a_j e_j(x)|^2< +\infty
 ~\text{for a.e.}~x\in\Omega.
\end{align}
Meanwhile, by Weyl's asymptotic formula for the eigenvalues of the Laplace operator (see for instance \cite[Theorem XIII.78, pp. 271]{SIMON2}), we find that
$
  \lim_{j\rightarrow+\infty} \eta_j j^{-2/n}  >0,
$
showing that $(\eta_j^{-n})_{j\geq 1} \in \ell^2$. This, together with \eqref{aj-ejx-l2-convergence} and the Cauchy--Schwarz inequality, yields
\begin{align*}
  \sum_{j\geq 1}   \eta_j^{-n} |a_j e_j(x)|
  \leq \Big(\sum_{j\geq 1} \eta_j^{-2n} \Big)^{ \frac{1}{2} }
  \Big(\sum_{j\geq 1} |a_j e_j(x)|^2 \Big)^{ \frac{1}{2} }
  <+\infty
  ~\text{for a.e.}~
  x\in \Omega.
\end{align*}
Then, by  \eqref{gj-fj-uniform-estimates}, we see that for a.e. $x\in\Omega$, both series
\begin{align*}
   \sum_{j\geq 1} \tilde f_j(t) \eta_j^{-n} a_j e_j(x)
   ~\text{and}~
   \sum_{j\geq 1} \tilde g_j(t) \eta_j^{-n} a_j e_j(x), ~t\in \mathcal O
\end{align*}
 absolutely converge over each nonempty compact subset $G\subset \mathcal O$, and thus, their sums are analytic over $\mathcal O$. Thus, by \eqref{solution-Galerkin-with-fj-and-L1bound}, the  assertion (A3) follows at once.   This completes the proof.
\end{proof}

\subsection{Exact observability}

Throughout this subsection, {\it we suppose that   $Q\subset\mathbb R^+ \times \Omega$ is  a nonempty measurable subset.} We first present the following observability estimate.

\begin{theorem}\label{thm-QST-by-yb-202110}
Let   $T>S\geq 0$.
 We suppose that $(Q,S,T)$ satisfies the MOC.
Then, for some $C>0$,
	\begin{eqnarray}\label{exact-ob-twosides-by-yb-202110}
	  \|y_0\|_{\mathcal H^{-4}}
	\leq C \int_{S}^{T}
    \| \chi_Q(t,\cdot) y(t,\cdot;y_0) \|_{L^2(\Omega)} 	 dt
	 ~~\mbox{for each}~ y_0\in L^2(\Omega).
	\end{eqnarray}
\end{theorem}

 We will prove  Theorem \ref{thm-QST-by-yb-202110}
 through  a modified three-step strategy that was originally designed for the observability of the wave equation (see, for instance, \cite{BLR92,RauchTaylor}). The proof requires  several technical lemmas. The first one presents estimates about the analytic functions and  the MOC.

\begin{lemma}\label{lem-estimates-MCC}
Let $T>0$ and let $f$ be a non-trivial real analytic function over $[0,T]$. Then, there are two constants $C,\beta>0$ (independent of $\Omega,\,Q$) such that for each $S\in[0,T)$,
\begin{align}\label{integral-bounded-from-below-by-yb-2021-10}
   \int_S^T \chi_Q(t,x) |f(t)|  dt
   \geq C \left( \int_S^T \chi_Q(t,x) dt \right)^{\beta+1}
   ~~\text{for a.e.}~
   x\in \Omega.
\end{align}
If we further assume that $(Q,S,T)$ satisfies the MOC, then
\begin{align}\label{MCC-epsilon-bounded-from-below-by-yb-2021-10}
   \underset{x\in\Omega}{\text{ess-inf}}
   \int_{S+\varepsilon}^T \chi_Q(t,x) |f(t)|  dt
   >0
   ~~\text{when}~
   0\leq  \varepsilon <\mathcal{T}_{\Omega}(Q,S,T),
\end{align}
where $\mathcal{T}_{\Omega}(Q,S,T)$ is given by \eqref{Q-S-T-bounded-from-below}.
\end{lemma}

\begin{proof}
Because  $f$ is real analytic, it
has at most a finite number of distinct zeros over $[0,T]$, denoted by $\{t_j\}_{j=1}^m$. Let
$d_j\in \mathbb N^+$ be the order of $t_j$ ($j=1,\ldots,m$). 
Then, $f(t) / \displaystyle\min_{1\leq j \leq m}  |t-t_j|^{d_j} $ can be extended to a continuous function over $[0,T]$ without zeros.  Thus, there is a $C_1>0$ such that
\begin{align*}
   |f(t)|  \geq C_1  \min_{1\leq j \leq m}  |t-t_j|^{d_j},
   ~~\text{when}~
   0\leq t \leq T.
\end{align*}
We set $\beta := \displaystyle\max_{1\leq j \leq m} d_j$.
Then, there is a $C_2>0$ such that
\begin{align}\label{kernel-M-bounded-from-below-by-yb-202110}
   |f(t)|  \geq C_1  \min_{1\leq j \leq m}  T^{d_j} \Big(|t-t_j|/T\Big)^{d_j}
   \geq C_2  \min_{1\leq j \leq m}  |t-t_j|^{\beta},\;\;\mbox{when}\;\;0\leq t \leq T.
\end{align}

We now show that \eqref{integral-bounded-from-below-by-yb-2021-10} is satisfied. To this end, we arbitrarily
 fix  $S\in[0,T)$.
For each $x\in \Omega$, we define  the following set:
\begin{align}\label{2.26,2-24}
  I(x):=\big\{t\in[S,T]~:~(t,x)\in Q \big\}.
\end{align}
Because   $Q\subset\mathbb R^+ \times \Omega$ is   measurable, we see that  for a.e. $x\in \Omega$, $I(x)$
is measurable. We arbitrarily fix  $x\in \Omega$ with $I(x)$ measurable. We define
\begin{align}\label{2.27,2-24}
  E_j(x) := \Big\{ t\in I(x)~:~ |t-t_j|^{\beta} =  \min_{1\leq k \leq m}  |t-t_k|^{\beta}
  \Big\},~~
  1\leq j \leq m.
\end{align}
The following facts hold:

 Fact 1: It follows  from \eqref{2.26,2-24}
and \eqref{kernel-M-bounded-from-below-by-yb-202110} that
\begin{align}\label{2.28,2-24}
\int_S^T \chi_Q(t,x) |f(t)|  dt &=
   \int_{I(x)}   |f(t)|  dt
   \geq C_2 \int_{I(x)} \min_{1\leq j \leq m}  |t-t_j|^{\beta} dt.
  \end{align}

  Fact 2: It follows from \eqref{2.26,2-24} and \eqref{2.27,2-24} that
   $I(x) = \displaystyle\bigcup_{j=1}^m E_j(x)$.  Thus, there is a $j_0\in\{1,\ldots,m\}$
   such that
\begin{align}\label{2.29,2-24}
   | E_{j_0}(x) | \geq  \frac{|I(x)|}{m}
   = \frac{1}{m}  \int_S^T \chi_Q(t,x) dt.
\end{align}

Fact 3: It follows
from \eqref{2.27,2-24} (where $j=j_0$) that
\begin{align}\label{2.30,2-24}
\int_{I(x)} \min_{1\leq j \leq m}  |t-t_j|^{\beta} dt
\geq
\int_{E_{j_0}(x)} |t- t_{j_0} |^{\beta} dt.
   \end{align}

   Fact 4: We have
   \begin{align}\label{2.31,2-24}
\int_{E_{j_0}(x)} |t- t_{j_0} |^{\beta} dt
\geq
\int_{|t-t_{j_0}| \leq \frac{| E_{j_0}(x) |}{2} }  |t- t_{j_0} |^{\beta} dt.
   \end{align}
To show that \eqref{2.31,2-24} is satisfied, we set
$E_{-}:=E_{j_0}(x) \cap (-\infty,t_{j_0})$ and $E_+ := E_{j_0}(x)\setminus E_{-}$.
We fix an open interval $I\Subset (-\infty,t_{j_0})$.
 Since the function
$t\mapsto |t-t_{j_0}|^\beta$ ($t\in \mathbb{R}$) is decreasing over
$(-\infty, t_{j_0})$, we have
\begin{align*}
\int_I |t- t_{j_0} |^{\beta} dt
   \geq \int_{t_{j_0}- |I|}^{ t_{j_0} } |t- t_{j_0} |^{\beta} dt.
   \end{align*}
Because any  subset of positive measure in $\mathbb{R}$ differs from a countable intersection of open sets by a null set, the above inequality, where $I$ is replaced by a subset of positive measure in $(-\infty, t_{j_0})$, still holds. In particular, we have
$
\int_{ E_{-} } |t- t_{j_0} |^{\beta} dt
\geq \int_{t_{j_0}- |E_{-}|}^{ t_{j_0} } |t- t_{j_0} |^{\beta} dt.
$
Similarly, we can show that
$
\int_{ E_{+} } |t- t_{j_0} |^{\beta} dt
\geq \int^{ t_{j_0}+ |E_{+}| }_{ t_{j_0} } |t- t_{j_0} |^{\beta} dt.
$
Thus, we have
\begin{equation}\label{2.32,2-24}
   \int_{E_{j_0}} |t- t_{j_0} |^{\beta} dt =
   \int_{ E_{-} } |t- t_{j_0} |^{\beta} dt  +  \int_{ E_{+} } |t- t_{j_0} |^{\beta} dt
   \geq \int_{t_{j_0}- |E_{-}|}^{ t_{j_0}+ |E_{+}| } |t- t_{j_0} |^{\beta} dt.
\end{equation}
Since $|E_{j_0}|=|E_+|+|E_-|$ and the function $t \mapsto|t- t_{j_0} |^{\beta}$
 ($t\in \mathbb{R}$) is symmetric about $t=t_{j_0}$, we obtain \eqref{2.31,2-24}
 from \eqref{2.32,2-24} at once.

  Fact 5: It follows from \eqref{2.29,2-24} that
\begin{align}\label{2.33,2-24}
\int_{|t-t_{j_0}| \leq \frac{| E_{j_0}(x) |}{2} }  |t- t_{j_0} |^{\beta} dt
\geq
\frac{2}{\beta+1} \left( \frac{| E_{j_0}(x) | }{2} \right)^{ \beta+1 }.
\end{align}

Hence, it follows from \eqref{2.28,2-24}, \eqref{2.30,2-24}, \eqref{2.31,2-24},
and \eqref{2.33,2-24} that
\begin{align*}
   \int_S^T \chi_Q(t,x) |f(t)|  dt
   \geq \frac{2C_2}{\beta+1} \left( \frac{1 }{2m}  \int_S^T \chi_Q(t,x) dt
   \right)^{ \beta+1 },
\end{align*}
which leads to \eqref{integral-bounded-from-below-by-yb-2021-10}.

Finally, we prove \eqref{MCC-epsilon-bounded-from-below-by-yb-2021-10}. To this end, we arbitrarily
 fix  $(Q,S,T)$ with the MOC and   $\varepsilon \in [0, \mathcal{T}_{\Omega}(Q,S,T))$. It is clear that
  \begin{align*}
    \int_{S+\varepsilon}^T \chi_Q(t,x) dt
    \geq \int_S^T \chi_Q(t,x) dt -\varepsilon
    ~~\text{for a.e.}~x\in \Omega.
  \end{align*}
  The above, along with \eqref{integral-bounded-from-below-by-yb-2021-10}
  (where $S$ is replaced by $S+\varepsilon$) and \eqref{Q-S-T-bounded-from-below}, leads to
 \eqref{MCC-epsilon-bounded-from-below-by-yb-2021-10}.
 This completes the proof.
\end{proof}

The following lemma corresponds to the first step of the above-mentioned three-step strategy, which is  a relaxed observability inequality.
\begin{lemma}\label{lem-relaxed-exact-ob}
	Let $T>S\geq 0$.
We assume that $(Q,S,T)$ satisfies the MOC.
Then, there is a $C>0$  such that
	\begin{eqnarray}\label{0425-high-frequency-ob-ineq-0}
	C \|y_0\|_{\mathcal H^{-4}}
	\leq \int_{S}^{T}
	\|\chi_Q(t,\cdot) y(t,\cdot;y_0) \|_{L^2(\Omega)}  dt
	+ \|y_0\|_{\mathcal H^{-6}}
	 \;\;\mbox{for each}\;\; y_0\in L^2(\Omega).
	\end{eqnarray}
	\end{lemma}

\begin{proof}
Since $(Q,S,T)$ satisfies the MOC,
 it follows from \eqref{Q-S-T-bounded-from-below} that
\begin{align}\label{epsilon-choice-by-yb-202110}
   \varepsilon_0:= \frac{1}{2} \mathcal{T}_{\Omega}(Q,S,T)
                > 0.
\end{align}
	We arbitrarily fix a   $y_0\in L^2(\Omega)$. We recall that
$y(t,\cdot;y_0)= \varPhi(t) y_0$.
By Corollary \ref{cor-only-h1} (with $\beta=3$),
	using
	the triangle inequality,
	we can find a  $C_1>0$ (independent of $z$)  such that
	\begin{align*}
		\|\chi_Q(t,\cdot) y(t,\cdot;y_0)\|_{L^2(\Omega)}
		\geq
		\|  \chi_Q(t,\cdot) M(t) A^{-2} y_0 \|_{L^2(\Omega)}
		-  C_1 t^{-3} \|A^{-3} y_0\|_{ L^2(\Omega) },~ t\in(S+\varepsilon_0,T).
	\end{align*}
 This yields
	\begin{align}\label{0425-high-frequency-ob-ineq-2}
	& \int_{S+\varepsilon_0}^{T}
	\| \chi_Q(t,\cdot) M(t) A^{-2} y_0 \|_{L^2(\Omega)} dt
\nonumber\\
	 \leq& \int_{S+\varepsilon_0}^{T}   \|\chi_Q(t,\cdot) y(t,\cdot;y_0)\|_{L^2(\Omega)}  dt
	+  C_1 (S+\varepsilon_0)^{-3}   \|y_0\|_{\mathcal H^{-6}}.
	\end{align}
Meanwhile, we have
	\begin{align*}
	  	\bigg\| \Big( \int_{S+\varepsilon_0}^T \chi_Q(t,\cdot)   |M(t)|  dt \Big)
	A^{-2} y_0  \bigg\|_{L^2(\Omega)}
    \leq  \int_{S+\varepsilon_0}^T
    \| \chi_Q(t,\cdot) M(t) A^{-2} y_0 \|_{L^2(\Omega)}  dt.
	\end{align*}
From this and \eqref{0425-high-frequency-ob-ineq-2}, we obtain
\begin{align}\label{leading-estimate-lower-bound-by-yb-202110}
	& \left( \int_{\Omega}
\left(  \int_{S+\varepsilon_0}^T \chi_Q(t,x)  | M(t) |  dt
\right)^2
|(A^{-2}y_0)(x)|^2 dx
\right)^{ \frac{1}{2} }
\nonumber\\
	 \leq& \int_{S+\varepsilon_0}^{T}   \|\chi_Q(t,\cdot) y(t,\cdot;y_0)\|_{L^2(\Omega)}  dt
	+  C_1 (S+\varepsilon_0)^{-3}   \|y_0\|_{\mathcal H^{-6}}.
\end{align}
At the same time,
 since $(Q,S,T)$ satisfies the MOC,
 it follows from \eqref{MCC-epsilon-bounded-from-below-by-yb-2021-10} (with $(f,\varepsilon)=(M,\varepsilon_0)$) 
  and \eqref{epsilon-choice-by-yb-202110}  that for a.e. $x\in\Omega$,
\begin{align*}
    \int_{S+\varepsilon_0}^T \chi_Q(t,x) |M(t)|  dt
\geq   \underset{x\in\Omega}{\text{ess-inf}} 
\left(  \int_{S+\varepsilon_0}^T \chi_Q(t,x) |M(t)|  dt 
\right)
   >0.
\end{align*}
	The above, along with \eqref{leading-estimate-lower-bound-by-yb-202110}, yields \eqref{0425-high-frequency-ob-ineq-0}. This  completes  the proof.
\end{proof}

The next lemma gives a  unique continuation property for the solutions to equation \eqref{our-obserble-system}.
 It corresponds to
the second step in the above-mentioned three-step strategy.

\begin{lemma}\label{lem-unique-continuation-for-low-frequence}
	Let $T>S\geq 0$ and let $y_0 \in \mathcal H^{-4}$.
 	We assume that $(Q,S,T)$ satisfies the MOC.
	If
	\begin{eqnarray}\label{varPhi-Q-zero}
	y(t,x;y_0)=0
	~~\mbox{for a.e.}~  (t,x)\in Q\cap \big((S,T)\times\Omega\big),
	\end{eqnarray}
	then $y_0=0$ in $\mathcal H^{-4}$.
\end{lemma}

\begin{proof}
First, note that $A^{-2}y_0\in L^2(\Omega)$ since
$y_0 \in \mathcal H^{-4}$. Second, notice that for a.e. $x\in \Omega$, $y(\cdot,x;y_0)$ vanishes over a subset of positive measure in $\mathbb{R}$. Indeed,
  since $(Q,S,T)$ satisfies the MOC,
we can use  \eqref{Q-S-T-bounded-from-below}
 to see that for a.e. $x\in \Omega$,
   the set   \begin{align*}
     I(x)
 :=\Big\{t>0~:~ (t,x)\in Q\cap \big((S,T)\times\Omega\big)
 \Big\}
\end{align*}
  is a subset of positive measure in $\mathbb{R}$. 
 This, along with \eqref{varPhi-Q-zero}, yields the second fact.

 Finally, from the above two facts, we can use
  Proposition \ref{pro-analyticity-pointwise}
  to determine that for a.e. $x\in \Omega$, $y(\cdot,x;y_0)\equiv 0$ over $(0,+\infty)$.
   Since $y_0\in \mathcal H^{-4}$,  $y(\cdot,\cdot;y_0)\in C((0,+\infty);L^2(\Omega))$ is obtained from Corollary \ref{cor-only-h1} (with $\beta=2$).
Therefore, $y(\cdot,\cdot;y_0)=0$   in $C((0,+\infty);L^2(\Omega))$.
          On the other hand,  given that $\{\varPhi(t)\}_{t\geq 0}$ is also a $C_0$ semigroup over $\mathcal H^{-4}$, we have
          \begin{align*}
            y_0(\cdot) = \lim_{t\rightarrow 0^+}  y(t,\cdot;y_0)
            ~\text{in}~  \mathcal H^{-4}.
          \end{align*}
      Then  $y_0=0$ in $\mathcal H^{-4}$ and this completes the proof.
\end{proof}

\begin{remark}\label{remark3.5}
In \cite{Zhou-Gao}  it was shown that, when  $M$ is an exponential function, if $y=0$ over
$(0,T)\times\omega$ (where $\omega$ is an open nonempty subset of $\Omega$), then
$y$ is identically zero.
 We do not know if  such a unique continuation property holds for general analytic memory kernels. But the unique continuation property shown above suffices for our purposes since we will be dealing with observation sets fulfilling the MOC condition.
\end{remark}

We are now in a position to prove Theorem \ref{thm-QST-by-yb-202110} combining the compactness--uniqueness argument
and  Lemmas \ref{lem-relaxed-exact-ob}--\ref{lem-unique-continuation-for-low-frequence}.

\begin{proof}[Proof of Theorem \ref{thm-QST-by-yb-202110}]
By contradiction, we suppose that   \eqref{exact-ob-twosides-by-yb-202110} is not true.
 Then, there is a $\{ \hat z_k \}_{k=1}^\infty \subset
	L^2(\Omega) $
	such that
	\begin{eqnarray}\label{0427-second-s3-thm3.2-ob-1}
	\| \hat z_k \|_{\mathcal H^{-4}}
	= 1\;\;\mbox{for each}\;\;k\in \mathbb N^+;
	\;\;\;\;
	\lim_{k\rightarrow\infty} \int_{S}^{T}    \|\chi_Q(t,\cdot) y(t,\cdot;\hat z_k) \|_{L^2(\Omega)}  dt
	= 0.
	\end{eqnarray}
	From the first equality in (\ref{0427-second-s3-thm3.2-ob-1}), we can find  a subsequence of $\{ \hat z_k \}_{k=1}^\infty$,  denoted in the same manner, and $\hat z \in \mathcal H^{-4} $ such that
	\begin{eqnarray}\label{0427-second-s3-thm3.2-ob-3}
	\hat z_k  \;\rightharpoonup\;   \hat z
	\;\;\mbox{weakly in}\;\;
	\mathcal H^{-4} ,
	\;\;\mbox{as}\;\;
	k\rightarrow\infty.
	\end{eqnarray}
 According to    Corollary \ref{cor-only-h1} (with $\beta=3$ and $s=0$) as well as \eqref{0427-second-s3-thm3.2-ob-3}, there is a subsequence of $\{ \hat z_k \}_{k=1}^\infty$,
denoted in the same manner, such that
\begin{eqnarray*}
y(\cdot,\cdot;\hat z_k) \;\rightharpoonup\;
 y(\cdot,\cdot;\hat z)\;\;\mbox{weakly in} \;\; L^2_{loc}((0,+\infty);L^2(\Omega)).
\end{eqnarray*}
This, along with  the second equality in (\ref{0427-second-s3-thm3.2-ob-1}), yields
	\begin{eqnarray}\label{2.43,2-25}
		y(\cdot,\cdot;\hat z) =0
		\;\;\mbox{over}\;\;
		Q \cap \big( (S,T)\times \Omega \big).
	\end{eqnarray}
	Since $\hat z \in \mathcal H^{-4} $ and   $(Q,S,T)$ satisfies the MOC,
 we can apply Lemma \ref{lem-unique-continuation-for-low-frequence}
and \eqref{2.43,2-25} to determine
that $\hat z=0 \;\;\mbox{in}\;\; \mathcal H^{-4}$.
		This, together with (\ref{0427-second-s3-thm3.2-ob-3}), implies that
	$\displaystyle\lim_{k\rightarrow\infty} \| \hat z_k  \|_{\mathcal H^{-6} }=0$.
		Thus, we can use  Lemma \ref{lem-relaxed-exact-ob}, as well as  the second equality in  (\ref{0427-second-s3-thm3.2-ob-1}), to find that
	$\displaystyle\lim_{k\rightarrow\infty} \| \hat z_k  \|_{\mathcal H^{-4}  }=0$,
	which contradicts the first equality in (\ref{0427-second-s3-thm3.2-ob-1}). Hence,  \eqref{exact-ob-twosides-by-yb-202110} is true. This ends the proof.
\end{proof}

\section{Proof of the main theorem}\label{sec-main-proofs}

 This section aims to prove Theorem \ref{202205TJyb-MainTheorem-merged}.

\begin{proof}[Proof of Theorem \ref{202205TJyb-MainTheorem-merged}]
We first aim to show that  ($i$) implies ($ii$). For this purpose, we assume that the statement $(i)$ is true, i.e., the triplet $(Q,S,T)$ satisfies the MOC. Then, by Theorem \ref{thm-QST-by-yb-202110},  we obtain the first inequality in \eqref{202205TJYB-TwoSideObservability}. To show the second inequality in \eqref{202205TJYB-TwoSideObservability},
we  apply  Corollary \ref{cor-only-h1} (with $\beta=2$ and $s=0$) to find a $C_1>0$ such that for each
$y_0\in L^2(\Omega)$,
\begin{align*}
    \int_S^T \|\varPhi(t) y_0 \|_{L^2(\Omega)} dt
    \leq& \int_S^T \|M(t) A^{-2} y_0 \|_{L^2(\Omega)}  dt
    +     \int_S^T   C_1 t^{-2} \|A^{-2} y_0 \|_{L^2(\Omega)}  dt
    \nonumber\\
    \leq& \left(  \int_S^T |M(t)|  dt
    +    C_1   \int_S^T t^{-2} dt \right)  \|A^{-2} y_0 \|_{L^2(\Omega)}.
\end{align*}
Since  $S>0$,
the second inequality in \eqref{202205TJYB-TwoSideObservability}
follows from the  inequality above.
Thus, \eqref{202205TJYB-TwoSideObservability} is proven, i.e., the statement $(ii)$ is true.

Next, we aim to  verify that ($ii$) implies ($i$).
We suppose that the statement $(ii)$ is true, i.e., \eqref{202205TJYB-TwoSideObservability} holds for some $C>0$.
It follows  from Corollary \ref{cor-only-h1} (with $\beta=3$ and $s=0$) that
for some  $C_2>0$,
\begin{align*}
    \| \varPhi(t)y_0 + M(t) A^{-2} y_0 \|_{L^2(\Omega)}
    \leq C_2 t^{-3} \| (-A)^{-3} y_0 \|_{L^2(\Omega)},
    ~y_0 \in L^2(\Omega),~t\in [S,T].
\end{align*}
Then, we   obtain from  \eqref{202205TJYB-TwoSideObservability}
that for each $y_0 \in L^2(\Omega)$,
\begin{align}\label{202210-TJ-YB-estimates-for-testing}
    &  \| A^{-2} y_0\|_{ L^2(\Omega) }
    =\| y_0\|_{ \mathcal H^{-4} }
    \leq C \int_S^T  \| \chi_Q(t,\cdot) \varPhi(t)y_0 \|_{L^2(\Omega)}  dt
    \nonumber\\
    \leq&   C  \|M\|_{C([0,T])}
    \int_S^T  \| \chi_Q(t,\cdot) A^{-2}y_0\|_{L^2(\Omega)}  dt
    +  C C_2 S^{-2}
    \| (-A)^{-3} y_0\|_{L^2(\Omega)}.
\end{align}
By a standard density argument,
we replace $A^{-2} y_0$ by $z$ in \eqref{202210-TJ-YB-estimates-for-testing} to determine  that for some $C_3>0$,
\begin{align}\label{estimate-for-necessary-geometry-by-yb-202210}
    C_3 \| z\|_{ L^2(\Omega) }
    \leq   \int_S^T  \| \chi_Q(t,\cdot) z\|_{L^2(\Omega)}  dt
    +  \| (-A)^{-1} z\|_{L^2(\Omega)}
    \;\;\mbox{for all}\;\; z\in L^2(\Omega).
\end{align}

Now, we will use \eqref{estimate-for-necessary-geometry-by-yb-202210} to derive that
\begin{align}\label{202205tjYB-QSubsetQmoc}
    (Q,S,T) ~\text{satisfies the MOC}.
\end{align}
In what follows, we use $B(x,r)$ to denote the closed ball in $\mathbb{R}^n$, centered at $x$ with radius $r$.
We arbitrarily fix an $x_0\in\Omega$ and set
\begin{align}\label{def-zk-for-geometry-by-yb-202110}
    z_k := |B(x_0,1/k)|^{- \frac{1}{2} }  \chi_{B(x_0,1/k)},~k\in \mathbb N^+.
\end{align}
It is clear that as $k\in \mathbb N^+$ is large,
\begin{align}\label{3.6--2-27}
    \text{supp}\,z_k\subset \Omega,~
    \|z_k\|_{L^2(\Omega)}=1
    ~\text{and}~
    z_k
    \;\rightharpoonup\;
    0~\text{weakly in}~ L^2(\Omega).
\end{align}
The last equality in \eqref{3.6--2-27} implies that $\displaystyle\lim_{k\rightarrow+\infty} (-A)^{-1} z_k=0$ in $L^2(\Omega)$.
This, along with \eqref{estimate-for-necessary-geometry-by-yb-202210} (where $z=z_k$) and the  second equality in \eqref{3.6--2-27}, yields
\begin{align}\label{C2-Lp-ball-by-yb-202110}
    C_3  \leq \limsup_{k\rightarrow+\infty}
    \int_0^T  \| \chi_Q(t,\cdot) z_k\|_{L^2(\Omega)}  dt.
    \end{align}
Applying  H\"{o}lder's inequality to \eqref{C2-Lp-ball-by-yb-202110} leads to
\begin{align*}
    C_3  \leq \limsup_{k\rightarrow+\infty}
    \left[
    \left( \int_S^T 1 dt \right)^{  \frac{1}{2} }
    \cdot
    \left( \int_S^T  \| \chi_Q(t,\cdot) z_k\|_{L^2(\Omega)}^2  dt
    \right)^{ \frac{1}{2} }
    \right].
    \end{align*}
This, along with \eqref{def-zk-for-geometry-by-yb-202110}, implies that
\begin{align*}
    C_3^2  T^{-1}
    \leq
    \limsup_{k\rightarrow+\infty}
    \left(
    \frac{1}{|B(x_0,1/k)|}
    \int_{B(x_0,1/k)} \bigg(\int_S^T\chi_Q(t,x)  dt \bigg) dx
    \right).
\end{align*}
Because $x_0\in \Omega$ was arbitrarily taken, the above inequality gives
\begin{align*}
    \int_S^T \chi_Q(t,x) dt  \geq  C_3^2 T^{-1}
    ~~\text{for a.e.}~  x \in \Omega.
\end{align*}
Thus, \eqref{Q-S-T-bounded-from-below} holds for the current  $(Q,S,T)$, which leads to \eqref{202205tjYB-QSubsetQmoc}, i.e., the statement $(i)$ is true.

\vskip 5pt
The proof of Theorem \ref{202205TJyb-MainTheorem-merged} is now completed.
\end{proof}

 \section{Extension of the main theorem }\label{sec-further-studies}

This section aims to extend Theorem \ref{202205TJyb-MainTheorem-merged} to the case with $S=0$  in the following two directions:
First, we show that the MOC satisfied by $(Q,0,T)$ is still a sufficient and necessary condition to ensure the two-sided observability inequality with a weight $t^\alpha$ where $\alpha>1$ (see Theorem \ref{202210-thm-weighted-observability} below). Second, we show that the MOC satisfied by $(Q,0,T)$ is  a sharp sufficient  condition to ensure the two-sided observability inequality without weight (see Theorem \ref{thm-null-observability-without-weight} below).

 \begin{theorem}\label{202210-thm-weighted-observability}
    Let  $T>0$.
    Then, for each $\alpha>1$ and each nonempty measurable subset $Q\subset (0,+\infty) \times \Omega$, the following two statements are equivalent:
    \begin{itemize}
        \item[(i)] The triplet $(Q,0,T)$ satisfies the MOC.
        \item[(ii)]  There is a constant $C>0$ such that
            \begin{align}\label{202205TJYB-TwoSideObservability-with-weight}
                \frac{1}{C} \| y_0\|_{ \mathcal H^{-4} }
                \leq  \int_0^T  \big\| \chi_Q(t,\cdot) y(t,\cdot;y_0) \big\|_{L^2(\Omega)}
                t^{\alpha}   dt
                \leq  C \| y_0\|_{ \mathcal H^{-4} }
                ~~~\mbox{for all}~ y_0 \in L^2(\Omega).
            \end{align}
    \end{itemize}
Furthermore, when $\alpha\leq 1$, (i) and (ii) are not equivalent if $Q$ contains a set of the form $(0,\varepsilon) \times B_{\varepsilon}$ (where $\varepsilon>0$ is small enough such that   $\Omega$ contains an open ball
   $B_{\varepsilon}$ of radius $\varepsilon$).
\end{theorem}

  \begin{theorem}\label{thm-null-observability-without-weight}
 Let $T>0$ and let $Q\subset (0,+\infty) \times \Omega$ be  a  nonempty measurable  subset. 
Then, for the following statements, the former leads to the latter:
\begin{itemize}
  \item[(i)] The triplet $(Q,0,T)$ satisfies the MOC.
  \item[(ii)] There is a $C>0$ such that
   \begin{align}\label{exact-observability-1003-WithoutWeight}
     \| y_0\|_{ \mathcal H^{-4} }
    \leq  C \| \chi_Q y(\cdot,\cdot;y_0) \|_{ L^1(0,T;L^2(\Omega)) }\;\;\mbox{for all}\;\;y_0 \in L^2(\Omega).
\end{align}
 \item[(iii)]  There is a $C>0$  such that
\begin{align}\label{null-observability-1003-without-weight}
   \| y(T,\cdot;y_0) \|_{L^2(\Omega)}
   \leq C    \| \chi_Q y(\cdot,\cdot;y_0) \|_{ L^1(0,T;L^2(\Omega)) }\;\;\mbox{for all}\;\;y_0 \in L^2(\Omega).
\end{align}

\item[(iv)]  The triplet $(Q,0,T)$ satisfies
    \begin{align}\label{202203-MCCofBall}
      \inf_{x_0\in\Omega}  {\int\hspace{-1.05em}-}_{ B(x_0,r) }
      \bigg( \int_0^T \chi_Q(t,x) dt
      \bigg) dx >0
       ~~\text{for each}~ r>0,
    \end{align}
    where ${\int\hspace{-0.9em}-}_{ B(x_0,r) }$ denotes the average value of the integral over the closed ball  $B(x_0,r)\subset\mathbb R^n$, centered at $x_0$ with radius $r$.
\end{itemize}

\end{theorem}

\begin{remark}\label{202205tjYB-remark-ObWithoutWeight}
 $(i)$  Theorem \ref{202210-thm-weighted-observability}
  shows that the equivalence in Theorem \ref{202205TJyb-MainTheorem-merged} remains true
  for the case with  $S=0$ by   inserting the weight function $t^\alpha$ (with $\alpha>1$) into the integrand
      in \eqref{202205TJYB-TwoSideObservability}.

Theorem \ref{202210-thm-weighted-observability} is mainly motivated by the following facts. First, the wave-like effect in  equation \eqref{our-obserble-system} determines the geometry of the observable set $Q$ and the space $\mathcal H^{-4}$ for the initial data. Second, the condition that $\alpha>1$ is
determined by  the heat-like effect in equation \eqref{our-obserble-system} (see Remark \ref{20221020-TjYb-remark-on-NecessaryWeight} for more explanations).

 $(ii)$ Theorem \ref{thm-null-observability-without-weight}
 provides a sharp sufficient geometric  condition (i.e., $(Q,0,T)$ satisfies the MOC in \eqref{Q-S-T-bounded-from-below})
 ensuring the observability
 inequalities \eqref{exact-observability-1003-WithoutWeight}
 and \eqref{null-observability-1003-without-weight}.

   The gap between $(i)$ and $(iv)$ in this theorem is rather thin for the following reasons (see also Example \ref{example-mcc} for further details on these two conditions):
   first, $(i)$ implies $(iv)$ directly; second, $(i)$ ensures that the  time for almost all characteristic lines to pass through the observation set
   has a positive lower bound (see also $(i)$ of Remark \ref{remark1.1,2-22}); third, $(iv)$ states that the aforementioned time, averaged over any small ball of a fixed radius, has a positive lower  bound.
\end{remark}

\begin{example}\label{example-mcc}
    Here we give concrete examples about the MOC (given by \eqref{Q-S-T-bounded-from-below}) and its average version  \eqref{202203-MCCofBall}. Let $\Omega:=(0,1)$ and $\varepsilon\in(0,1)$. We define the following two subsets in $\mathbb R^+  \times (0,1)$:
    \begin{align*}
        Q_1:=&\Big\{(t,x)\in \mathbb R^+  \times (0,1)
        ~:~   f(x;1) < t <  f(x;1) + \frac{\varepsilon}{2}
        \Big\},
        \nonumber\\
        Q_2:=&\Big\{(t,x)\in \mathbb R^+  \times (0,1)
        ~:~           f(x;1)  < t <  f(x;1+\varepsilon)
        \Big\},
    \end{align*}
    where  for each $a\in \mathbb R$,
    \begin{align*}
        f(x;a):=\begin{cases}
            a x,~x\in(0,1/2),\\
            a(1-x),~x\in[1/2,1).
        \end{cases}
          \end{align*}
    See Figure \ref{202310-yb-ComparationBetweenMOCandAveragedMOC} for their images (where $\varepsilon=0.1$). In this figure, the red set  is $Q_1$ and the blue set  is $Q_2$.
    One can directly check that  $(Q_1,0,1)$ satisfies both the MOC and the condition \eqref{202203-MCCofBall}, and $(Q_2,0,1)$ satisfies the condition \eqref{202203-MCCofBall} but not the MOC.
    \begin{figure}
        \centering
        \includegraphics[height=6cm,width=12cm]{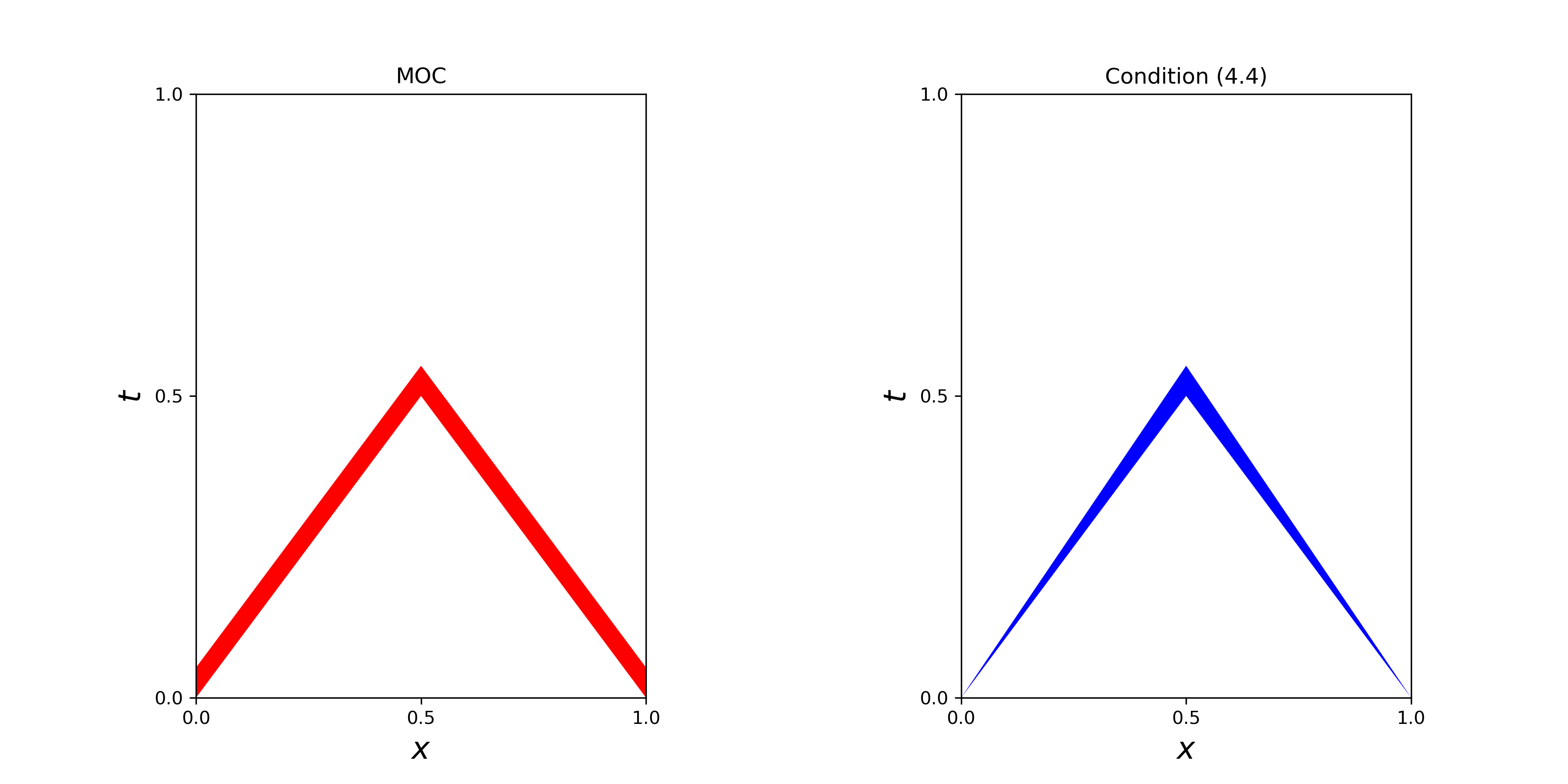}
        \caption{The red and blue sets correspond to the MOC and the condition \eqref{202203-MCCofBall}, respectively.}
        \label{202310-yb-ComparationBetweenMOCandAveragedMOC}
    \end{figure}
\end{example}
\color{black}

   The proof of  Theorem \ref{202210-thm-weighted-observability}
needs the following proposition
 whose proof is put in Section \ref{Appendix-SeveralProofs} in the Appendix.

\begin{proposition}\label{202202yb-prop-DoubleSidesOb-NecessaryWeights}
      Let $Q\supset (0,\varepsilon_0)\times\omega$ be  measurable, where $\varepsilon_0>0$ and $\omega\subset\Omega$ is a nonempty  open subset. Let $\alpha\in \mathbb R$.  Assume that the second inequality in \eqref{202205TJYB-TwoSideObservability-with-weight} holds for the above $Q$, i.e.,
     there is a constant $C>0$ such that
    \begin{align}\label{202210TJYB-TwoSideObservability-necessary-weight}
           \int_0^T  \big\| \chi_Q(t,\cdot) y(t,\cdot;y_0) \big\|_{L^2(\Omega)}
        t^{\alpha}   dt
        \leq  C \| y_0\|_{ \mathcal H^{-4} }
        ~~\mbox{for all}\;\;y_0 \in L^2(\Omega).
    \end{align}
      Then, it holds that $\alpha>1$.
\end{proposition}

\begin{remark}\label{20221020-TjYb-remark-on-NecessaryWeight}
The condition $\alpha>1$ in
Proposition \ref{202202yb-prop-DoubleSidesOb-NecessaryWeights}
is essentially determined by the heat-like effect in equation \eqref{our-obserble-system}. In fact, generally speaking, we cannot obtain that the pure heat solution $ e^{t A} y_0$ ($t>0$) is in the space $L^1(0,T;L^2(\Omega))$ from the fact that  $y_0 \in \mathcal H^{-4}$, unless this solution is multiplied by a weight $t^{\alpha}$ with $\alpha>1$. The same can be said about the solution $\varPhi(t)y_0$ ($t>0$),  due to its heat-like nature. This is the essential idea in the proof of Proposition \ref{202202yb-prop-DoubleSidesOb-NecessaryWeights}.

At last, $\alpha=1$ is critical because $\int_0^T e^{-\lambda t} t^{\alpha} dt \sim \Gamma(\alpha+1) \lambda^{-\alpha-1}$ as $\lambda \rightarrow +\infty$  when $\alpha\geq 0$ (here $\Gamma(\cdot)$ is the Gamma function defined by Euler's integral of second kind), and we have
\begin{align*}
    \int_0^T e^{tA} y_0  \, t^{\alpha} dt
    ~\sim~  \Gamma(\alpha+1)   \,   (-A)^{-\alpha-1} y_0
\end{align*}
for any $y_0$ involving only high frequency spectral components. Then $\alpha=1$ is crucial when the above function $(-A)^{-\alpha-1} y_0$ is compared to the recovered term $(-A)^{-2} y_0$ in \eqref{202205TJYB-TwoSideObservability-with-weight}. A similar idea can be applied to the integral
$\int_0^T \| \chi_Q e^{tA} y_0 \|_{L^2(\Omega)} t^{\alpha} dt$,
 as it is actually done in the proof of Proposition \ref{202202yb-prop-DoubleSidesOb-NecessaryWeights}.
\end{remark}

We are now in the position to prove Theorem \ref{202210-thm-weighted-observability}.

\begin{proof}[Proof of Theorem \ref{202210-thm-weighted-observability}]
     The proof is divided into the following two parts.

     \vskip 5pt
     \noindent \textit{\textbf{Part 1.} We prove $(i)\Leftrightarrow(ii)$ when $\alpha>1$.}

     Let  $\alpha>1$. We first  prove that $(i)\Rightarrow(ii)$. For this purpose, we suppose that $(i)$ is true, i.e., $(Q,0,T)$ satisfies the MOC.
    Then, from  \eqref{Q-S-T-bounded-from-below}
   it follows that
    \begin{align*}
        \varepsilon_0:= \frac{1}{2} \underset{x\in\Omega}{\text{ess-inf}} \int_0^T \chi_Q(t,x) dt\in (0,T).
    \end{align*}
    This implies that for a.e. $x\in \Omega$,
    \begin{align*}
        \int_{\varepsilon_0}^T \chi_Q(t,x) dt
        \geq \int_{0}^T \chi_Q(t,x) dt - \varepsilon_0
        \geq \frac{1}{2} \underset{x\in\Omega}{\text{ess-inf}} \int_0^T \chi_Q(t,x) dt
        >0,
    \end{align*}
    which shows that   $(Q,\varepsilon_0,T)$ also satisfies the MOC.
    Then, by Theorem \ref{thm-QST-by-yb-202110},  we obtain  \eqref{exact-ob-twosides-by-yb-202110} with $S=\varepsilon_0$. Meanwhile,
    it is clear that for each  $y_0\in L^2(\Omega)$,
    \begin{align*}
        \int_{\varepsilon_0}^{T}
        \| \chi_Q(t,\cdot) \varPhi(t) y_0 \|_{L^2(\Omega)} 	 dt
        \leq            \varepsilon_0^{ -\alpha }
        \int_{\varepsilon_0}^{T}
        \| \chi_Q(t,\cdot) \varPhi(t) y_0 \|_{L^2(\Omega)} t^{\alpha}	 dt.
        %\;\;\mbox{for all}\;\;y_0\in L^2(\Omega).
    \end{align*}
    This, along with   \eqref{exact-ob-twosides-by-yb-202110} (where $S=\varepsilon_0$),   leads to
    the first inequality in \eqref{202205TJYB-TwoSideObservability-with-weight}.
    To show the second inequality in \eqref{202205TJYB-TwoSideObservability-with-weight},
    we  apply  Corollary \ref{cor-only-h1} (with $\beta=2$ and $s=0$) to find a $C_1>0$ such that for each
    $y_0\in L^2(\Omega)$,
    \begin{align*}
        \int_0^T \|\varPhi(t) y_0 \|_{L^2(\Omega)} t^{\alpha}dt
        \leq& \int_0^T \|M(t) A^{-2} y_0 \|_{L^2(\Omega)} t^{\alpha} dt
        +     \int_0^T   C_1 t^{-2} \|A^{-2} y_0 \|_{L^2(\Omega)}  t^{\alpha} dt
        \nonumber\\
        \leq& \left(  \int_0^T |M(t)| t^{\alpha} dt
        +    C_1   \int_0^T t^{\alpha-2} dt \right)  \|A^{-2} y_0 \|_{L^2(\Omega)}.
    \end{align*}
    Since  $\alpha>1$,
    the second inequality in \eqref{202205TJYB-TwoSideObservability-with-weight}
    follows from the  inequality above.
    In conclusion, \eqref{202205TJYB-TwoSideObservability-with-weight} is proven, i.e., the statement $(ii)$ is true.

   Next, we aim to show that $(ii)\Rightarrow(i)$.   To this end, we assume that $(ii)$ is true, i.e., \eqref{202205TJYB-TwoSideObservability-with-weight} holds for some $C>0$.
    Set
    \begin{align}\label{beta-range-for-thm1-by-yb-202110}
        \beta_0:=\min \Big\{2+ (\alpha-1)/2,3  \Big\}
        \in (2,3]\cap (2,\alpha+1 )
    \end{align}
    (here, the fact that $\alpha>1$ is used). It follows  from Corollary \ref{cor-only-h1} (with $\beta=\beta_0$ and $s=0$) that
    for some  $C_2>0$,
    \begin{align}\label{varPhi-M(t)-remainder-by-by-202110}
        \| \varPhi(t)y_0 + M(t) A^{-2} y_0 \|_{L^2(\Omega)}
        \leq C_2 t^{-\beta_0} \| (-A)^{-\beta_0} y_0 \|_{L^2(\Omega)},
        ~y_0 \in L^2(\Omega),~t\in(0,T].
    \end{align}
    Because $2<\beta_0<\alpha+1$ (see \eqref{beta-range-for-thm1-by-yb-202110}), we   obtain from  \eqref{202205TJYB-TwoSideObservability-with-weight}  and \eqref{varPhi-M(t)-remainder-by-by-202110}
    that for each $y_0 \in L^2(\Omega)$,
    \begin{align*}%\label{estimates-for-testing-by-yb-202110}
        &  \| A^{-2} y_0\|_{ L^2(\Omega) }
        =\| y_0\|_{ \mathcal H^{-4} }
        \leq C \int_0^T  \| \chi_Q(t,\cdot) \varPhi(t)y_0 \|_{L^2(\Omega)}
        t^{\alpha} dt
        \nonumber\\
        \leq&   C T^{ \alpha } \|M\|_{C([0,T])}
        \int_0^T  \| \chi_Q(t,\cdot) A^{-2}y_0\|_{L^2(\Omega)}  dt
        +  C C_2 \frac{ T^{ \alpha+1-\beta_0 } }{\alpha+1-\beta_0 }
        \| (-A)^{-\beta_0} y_0\|_{L^2(\Omega)}.
    \end{align*}
    Then, by a standard density argument,
    we obtain that for some $C_3,\beta>0$,
    \begin{align}\label{estimate-for-necessary-geometry-by-yb-202110}
        C_3 \| z\|_{ L^2(\Omega) }
        \leq   \int_0^T  \| \chi_Q(t,\cdot) z\|_{L^2(\Omega)}  dt
        +  \| (-A)^{-\beta} z\|_{L^2(\Omega)}
        \;\;\mbox{for all}\;\; z\in L^2(\Omega).
    \end{align}
   Now, we observe that \eqref{estimate-for-necessary-geometry-by-yb-202110} is similar to \eqref{estimate-for-necessary-geometry-by-yb-202210}. In a similar way as we used \eqref{estimate-for-necessary-geometry-by-yb-202210} to prove \eqref{202205tjYB-QSubsetQmoc}, we can deduce from  \eqref{estimate-for-necessary-geometry-by-yb-202110}  that \eqref{202205tjYB-QSubsetQmoc} holds with $S=0$  (i.e., $(Q,0,T)$ satisfies the MOC). Thus, the statement $(i)$ is proven.

    \vskip 5pt
    In conclusion, when $\alpha>1$, the equivalence between $(i)$ and $(ii)$ is proven.

\vskip 5pt
\noindent \textit{\textbf{Part 2.} We prove that when $\alpha\leq 1$,
    (i)  and (ii) are not equivalent if $Q$ contains a set of the form $ (0,\varepsilon) \times B_{\varepsilon} $, where $B_{\varepsilon}\subset \Omega$ is  a small ball.
}

Let $\alpha\leq 1$.  First of all, the statement $(ii)$ can not hold. Otherwise, since $Q \supset (0,\varepsilon) \times B_{\varepsilon} $, by the second inequality in $(ii)$, we can apply  Proposition \ref{202202yb-prop-DoubleSidesOb-NecessaryWeights} to determine that $\alpha > 1$, which leads to a contradiction (since it was assumed that $\alpha \leq 1$).

Next, when $(Q,0,T)$ satisfies the MOC and $Q$ contains the form $ (0,\varepsilon) \times B_{\varepsilon} $ (the existence of such $Q$ is easily guaranteed), the statement $(i)$ holds but the statement $(ii)$ does not. Thus, $(i)$  and $(ii)$ are not equivalent if $Q$ contains the form $ (0,\varepsilon) \times B_{\varepsilon} $.

    \vskip 5pt
    In conclusion,  we have completed   the proof of  Theorem \ref{202210-thm-weighted-observability}.
\end{proof}

 Next, we will  prove Theorem \ref{thm-null-observability-without-weight}.

\begin{proof}[Proof of Theorem \ref{thm-null-observability-without-weight}]
    We arbitrarily fix a nonempty measurable subset $Q\subset \mathbb{R}^+\times \Omega$. We organize
    the proof in several steps.

    \vskip 5pt
   \noindent {\it Step 1.  We show $(i)\Rightarrow(ii)$.}

 We assume that $(i)$ holds. Then, by Theorem \ref{202210-thm-weighted-observability}, we have \eqref{202205TJYB-TwoSideObservability-with-weight} with  $\alpha=2$.
 This, along with  the fact that $\displaystyle\sup_{0<t<T} t^{\alpha} =T^{\alpha}$, yields  $(ii)$ of Theorem \ref{thm-null-observability-without-weight}.

\vskip 5pt
  \noindent  {\it Step 2.  We show $(ii)\Rightarrow(iii)$.}

  According to Corollary \ref{cor-only-h1} (with $\beta=2$, $s=0$, and $t=T$), there is a $C_1>0$ such that
\begin{align*}
  \| \varPhi(T) y_0 \|_{L^2(\Omega)}
  \leq C_1 \| A^{-2} y_0 \|_{L^2(\Omega)}
  = C_1 \| y_0 \|_{ \mathcal H^{-4} }.
\end{align*}
This, along with  $(ii)$ of this theorem, implies $(iii)$ of this theorem.

\vskip 5pt
 \noindent{\it Step 3.  We show $(iii)\Rightarrow(iv)$.}

 Let $(iii)$ of this theorem hold.
    By contradiction, we suppose that
  $(iv)$ is not true. Then,
   there is an $r_1>0$ such that %\eqref{202203-MCCofBall}
 \begin{align}\label{3.8-3-20}
      \inf_{x_1\in\Omega}  {\int\hspace{-1.05em}-}_{ B(x_1,r_1) }
      \bigg( \int_0^T \chi_Q(t,x) dt
      \bigg) dx =0.
 \end{align}
 We define the following function over $\overline{\Omega}$ via
 \begin{align*}
 F(x_1):= {\int\hspace{-1.05em}-}_{ B(x_1,r_1) }
      \bigg( \int_0^T \chi_Q(t,x) dt
      \bigg) dx,\;\;x_1\in  \overline{\Omega}.
 \end{align*}
 We can directly check that $F$ is a non-negative and continuous function
 over $\overline{\Omega}$.
 Then, by \eqref{3.8-3-20}, there is a minimizer $\hat x_1\in\overline{\Omega}$  such that
 $F(\hat x_1)=0$. From this,
  we can easily show that
  there is an $x_0 \in  \Omega$  and an $r>0$ such that
\begin{eqnarray}\label{sharp-geometry-pf-1}
\chi_Q(t,x) =0 ~~\text{for a.e.}~
(t,x) \in (0,T) \times B(x_0, r).
%	Q \cap \big( [0,T]\times\Omega \big)
%	\subset  [0,T] \times \big( \Omega\setminus B(x_0,r) \big).
\end{eqnarray}
 At the same time, we can apply $(iii)$ (in this theorem) to find  some $C>0$ such that
\begin{align*}
 \|\varPhi(T)y_0\|_{L^2(\Omega) }
 \leq C \| \chi_Q \varPhi(\cdot)y_0 \|_{ L^\infty(0,T;L^2( \Omega)) }
 \;\;\mbox{for each}\;\; y_0\in L^2(\Omega).
\end{align*}
This, together with  \eqref{sharp-geometry-pf-1}, yields
		\begin{eqnarray}\label{sharp-geometry-pf-s1-2}
		\|\varPhi(T)z\|_{L^2(\Omega) }
		\leq  C \| \varPhi(\cdot)z\|_{ L^\infty(0,T;L^2( \Omega\setminus B(x_0,r) )) }\;\;\mbox{for each}\;\; z\in L^2(\Omega).
		\end{eqnarray}

	Now, we will  present a contradiction to \eqref{sharp-geometry-pf-s1-2} by choosing a suitable sequence $\{z_k\}_{k\geq 1}$ in $L^2(\Omega)$.
	According to
 Theorem \ref{cor-0423-demcomposition} (in the Appendix),  $\{h_l(T)\}_{l\geq 1}$ (given by (\ref{thm-ODE-meomery-asymptotic-estimate-hypobolic}))
is not the zero sequence. Thus, we have
	\begin{align}\label{hj-nonzero}
	J := \min\Big\{ l\in \mathbb N^+~:~h_l(T)\neq 0 \Big\} < +\infty.
	\end{align}
	We select $\{\varepsilon_k\}_{k\geq 1}\subset (0,r/2)$
	and $\rho\in C_0^\infty(\mathbb R^n)$ such that
	\begin{eqnarray}\label{sharp-geometry-pf-3}
	\lim_{k\rightarrow\infty} \varepsilon_k=0;\;\;\;\;
	\| \rho \|_{L^2(\mathbb R^n)}=1;\;\;\;\;
	\rho(x)=0,~|x|_{\mathbb R^n}\geq 1.
	\end{eqnarray}
	We define $\{w_k\}_{k\geq 1} \subset C_0^\infty(\Omega)$ by
	\begin{eqnarray*}\label{sharp-geometry-pf-s1-4}
		w_k (x) := \varepsilon_k^{-n/2} \rho
		\Big(\frac{x-x_0}{\varepsilon_k} \Big),~x\in \Omega.
	\end{eqnarray*}
	From this and (\ref{sharp-geometry-pf-3}), we can directly check that
	\begin{eqnarray}\label{sharp-geometry-pf-s1-6}
	\mbox{supp}\, w_k \subset  B(x_0,r/2), ~\forall\; k\in \mathbb N^+;\;\|w_k\|_{L^2(\Omega)}=1,  ~\forall\; k\in \mathbb N^+;\;
	\mbox{w-}\lim_{k\rightarrow\infty} w_k=0
	\;\mbox{in}\;
	L^2(\Omega).
	\end{eqnarray}
	Now we define the sequence $\{z_k\}_{k\geq 1}\subset L^2(\Omega)$  by
	\begin{align}\label{definition of zk,7.26}
	z_k:= A^{J+1} w_k,\;\;k\geq 1.
	\end{align}
	From \eqref{definition of zk,7.26} and \eqref{sharp-geometry-pf-s1-6}, as well as \eqref{selfadjoit-elliptic-operator},  we see
	that for each $l\in\{0,1,\ldots,J+1\}$,
	\begin{eqnarray}\label{sharp-geometry-pf-s1-4-2-11}
	\mbox{supp}\,(A^{-l}  z_k)
	\subset B(x_0,r/2),~~\forall\, k\in \mathbb N^+;
	\;\;\;\;
	\lim_{k\rightarrow\infty}    z_k=0
	\;\;\mbox{in}\;\;
	\mathcal H^{-2J-4}.
	\end{eqnarray}
	Meanwhile,  by Theorem \ref{cor-0423-demcomposition}  (in the Appendix) with $N=J+1$, we have
	\begin{align}\label{varPhi-j+1-expression}
	\varPhi(t)=& e^{tA} \Big(Id + \sum_{l=0}^J p_l(t) (-A)^{-l-1} \Big)
	+ \sum_{l=1}^J h_l(t) (-A)^{-l-1} +  R_{J+1}(t,-A) (-A)^{-J-2}
	\nonumber\\
	:=& \mathcal P(t)  +  \mathcal W(t) + \widetilde{\mathcal R}(t),\;\;t>0,
	\end{align}
	where  $R_{J+1}(\cdot,-A) \in C( \mathbb R^+; \mathcal L(\mathcal H^0) )\bigcap L^\infty_{loc}(\overline{\mathbb R^+}; \mathcal L( \mathcal H^0))$ is given in Theorem \ref{cor-0423-demcomposition} with $N=J+1$.

With regard to three terms on the right hand side of \eqref{varPhi-j+1-expression}, we have the following:
	First, from the second conclusion in \eqref{sharp-geometry-pf-s1-4-2-11}, as well as the regularity of
	$R_{J+1}(\cdot,-A)$,  we find that
	\begin{align}\label{remainder-zk-0}
	\lim_{k\rightarrow\infty}  \sup_{0< t\leq T}
	\|\widetilde{\mathcal R} (t) z_k\|_{ L^2(\Omega) }
	=0.
	\end{align}
	Second, by the smoothing effect of $\{e^{tA}\}_{t\geq 0}$, we determine that
	\begin{align}\label{5.30,7.26w}
	\lim_{k\rightarrow\infty}  \mathcal P(T) z_k
	=0 \;\;\mbox{in}\;\;  L^2(\Omega).
	\end{align}
	Third,  by \eqref{hj-nonzero}, we have
	\begin{align*}
	\mathcal W(T) = h_J(T) (-A)^{-J-1},
	\end{align*}
	which, along with \eqref{definition of zk,7.26} and the second equality in \eqref{sharp-geometry-pf-s1-6},
	yields
	\begin{align}\label{5.31,7.26w}
	\| \mathcal W(T)z_k \|_{ L^2(\Omega) }=|h_J(T)|\neq0,~~\forall\,k\in \mathbb N^+.
	\end{align}
	Now from  \eqref{varPhi-j+1-expression}, \eqref{remainder-zk-0}, \eqref{5.30,7.26w},
	and \eqref{5.31,7.26w}, it follows that
	\begin{eqnarray}\label{sharp-geometry-T-zk-0}
	\lim_{k\rightarrow\infty} \|\varPhi(T)z_k\|_{L^2(\Omega)}
	= |h_J(T)|\neq 0.
	\end{eqnarray}

	Finally,    by \eqref{sharp-geometry-pf-s1-4-2-11} and the iterative use of Lemma \ref{lem-heat-up-to-0} (with $z=A^{-l}z_k$, $l\in\{0,\cdots,J+1\}$), we find that
	\begin{eqnarray}\label{P-outside-0}
	\lim_{k\rightarrow\infty} \sup_{0< t\leq T} \| \mathcal P(t)  z_k\|_{L^2(\Omega\setminus B(x_0,r))}
	=0.
	\end{eqnarray}
	Meanwhile, from the first conclusion in \eqref{sharp-geometry-pf-s1-4-2-11} and the definition of $\mathcal W(t)$ (see \eqref{varPhi-j+1-expression}), we see  that for each $k\in \mathbb N^+$,
	\begin{align}\label{5.34,7.26w}
	\mathcal W(t)z_k=0 \;\;\mbox{over}\;\;
	\Omega\setminus B(x_0,r),~~0\leq t\leq T.
	\end{align}
	From \eqref{varPhi-j+1-expression}, \eqref{P-outside-0}, \eqref{5.34,7.26w}, and \eqref{remainder-zk-0},
	we obtain
	\begin{eqnarray}\label{sharp-geometry-pf-s1-5}
	\lim_{k\rightarrow\infty}  \sup_{0< t\leq T} \| \varPhi(t)z_k \|_{L^2(\Omega\setminus B(x_0,r))}
	=0.
	\end{eqnarray}
	Now, the combination of \eqref{sharp-geometry-pf-s1-5}  and \eqref{sharp-geometry-T-zk-0}
	contradicts  (\ref{sharp-geometry-pf-s1-2}). Thus, $(iv)$  is true.

\vskip 5pt
	Finally, the proof of  Theorem \ref{thm-null-observability-without-weight} is completed.
\end{proof}

\section{Applications to control problems}\label{sec-controllability}

In this section, we denote by $Q\subset \mathbb{R}^+\times\Omega$  a nonempty measurable subset of positive measure that will play the role of support of the control, and let $p\in[1,+\infty]$.
We consider
the following  controlled heat equation with  memory:
\begin{eqnarray}\label{our-system}
\left\{
\begin{array}{ll}
\partial_t y(t,x)  - \Delta y(t,x)
+ \displaystyle\int_0^t M(t-s) y(s,x)ds =\chi_Q(t,x) u(t,x),
~&(t,x)\in \mathbb R^+ \times \Omega,\\
y(t,x)=0,~&(t,x)\in \mathbb R^+\times \partial\Omega,  \\
y(0,x)=y_0(x),~&\quad~~ x \in \Omega,
\end{array}
\right.
\end{eqnarray}
where  $y_0\in L^2(\Omega)$ and $u\in L^p(\mathbb{R}^+;L^2(\Omega))$.  We treat the solution of the control system \eqref{our-system} as a function from $[0,+\infty)$ to $L^2(\Omega)$ and denote it by  $y(\cdot;y_0,u)$.

Inspired by  the classical null controllability property of semigroups,
it is natural to address the following controllability problem: Given a $T>0$,  for each $y_0 \in L^2(\Omega)$, show  the existence of a control $u\in L^2(0,+\infty; L^2(\Omega))$, with $u=0$ over $(T,+\infty)$, such that
\begin{align}\label{20231013-yb-LongTimeEquivibrillum}
    y(t;y_0,u) = 0 ~~\text{for each}~ t\geq T.
\end{align}
We refer to \eqref{20231013-yb-LongTimeEquivibrillum} as the memory-type controllability property.

We were not able to solve this problem so far, and we turn our attention to the weaker goal of controlling   the state at time $t=T$, i.e.,
\begin{equation}\label{partial-control}
    y(T;y_0,u) = 0,
\end{equation}  
instead of the whole trajectory for $t \ge T$, i.e.,  $y(\cdot;y_0,u)|_{[T,+\infty)}$.

This, i.e., \eqref{partial-control},  constitutes a partial controllability problem. Indeed, even if
$y(T;y_0,u)\equiv 0$, due to the memory effects of the system, \eqref{20231013-yb-LongTimeEquivibrillum} will not be guaranteed.
This is the main difference and added difficulty of the control of heat-like equations involving memory terms.

In the sequel we limit the discussion to the partial controllability \eqref{partial-control} of $y(T;y_0,u)$. The analysis of the full control of $y(\cdot;y_0,u)|_{[T,+\infty)}$ constitutes an interesting open problem. Note that the methods in \cite{Chaves-Silva-Zhang-Zuazua}, imposing stronger  geometric conditions on the control sets and restricting the analysis to specific kernels, like polynomial ones, in particular, allow to ensure the full control of the system. Whether these results can be extended to  general analytic kernels under the sharp MOC of this paper is an interesting open problem. We will further discuss this issue in the next section.

In what follows, we present several  applications of  Theorem \ref{202205TJyb-MainTheorem-merged} (as well as Theorems \ref{202210-thm-weighted-observability} and \ref{thm-null-observability-without-weight})
to the control system \eqref{our-system}.
\subsection{Main results}

To state the first theorem, we introduce the following definitions:
\begin{itemize}[leftmargin=4em]
\item[(D3)] Given  $T>0$ and $y_0\in L^2(\Omega)$, let
\begin{align}\label{reachable-subspace-M}
	\mathcal R_M^p(T,y_0) := \big\{ y(T;y_0,u) ~:~
	u\in L^{p}(\mathbb R^+;L^2(\Omega))
	\big\}
	\end{align}
be the
	reachable set for the control system (\ref{our-system}) at time $T$.

\item[(D4)] Given  $T>0$ and $y_0\in L^2(\Omega)$, let
\begin{align*}
	\mathcal R_0^p(T,y_0) := \big\{ z(T;y_0,u) ~:~
	u\in L^{p}(\mathbb R^+;L^2(\Omega))
	\big\}
	\end{align*}
be the reachable set for the pure heat equation at time $T$, where $z(\cdot;y_0,u)$ is the  solution of the system  \eqref{our-system} with $M=0$.
\end{itemize}

The
first theorem refers to the  reachable set of the control system  \eqref{our-system}. It shows that, under the MOC, the reachable set of the control system \eqref{our-system} is the sum of the space $\mathcal H^4$ and the reachable set for the pure heat equation.

\begin{theorem}\label{thm-exact-controllability}
	Let $T>0$, $p\in[1,+\infty]$, and $y_0\in L^2(\Omega)$.	We assume that $(Q,0,T)$ satisfies the MOC.
 Then,
	\begin{align}\label{exact-ob-with-H8H6-target}
	\mathcal H^4 \subset \mathcal R_M^p(T,y_0) = \mathcal R_0^p(T,y_0) + \mathcal H^4.
	\end{align}
Moreover, for each $\alpha>1$,
\begin{align}\label{202202yb-H4-WeightedReachableSet}
  \mathcal H^4 =  \mathcal R_M^{\infty}(T,y_0,\alpha) :=
  \Big\{
      y(T;y_0,u)\in L^2(\Omega)~:~
      \underset{0<t<T}{\text{ess-sup}} \,
      \|(T-t)^{-\alpha} u(t)\|_{L^2(\Omega)}
      <+\infty
            \Big\}.
\end{align}
In particular, for every $y_0\in L^2(\Omega)$ and target $y_1\in  \mathcal H^4$, 
    there is a control $u\in L^{\infty}(\mathbb{R}^+;L^2(\Omega))$
    (with supp\,$u\subset[0,T]\times\Omega$)
    such that $y(T;y_0,u)=y_1$.
\end{theorem}

Furthermore, the MOC satisfied by $(Q,0,T)$ is sharp to ensure the structure \eqref{exact-ob-with-H8H6-target}. The following theorem, provides a necessary geometric condition on the control region for the partial controllability of the control system \eqref{our-system} to hold. The gap between the  MOC satisfied by $(Q,0,T)$ and this necessary condition is thin, as discussed in Remark \ref{202205tjYB-remark-ObWithoutWeight}.

\begin{theorem}\label{thm-necessity-geometry}
Let  $T>0$ and $p\in[1, +\infty]$.
We assume that the system \eqref{our-system} is
partially controllable over $[0,T]$ with $L^p$ controls
(i.e., for each $y_0\in L^2(\Omega)$, there is a $u\in L^p(\mathbb R^+; L^2(\Omega))$ such that $y(T;y_0,u)=0$). Then, the condition \eqref{202203-MCCofBall} holds.
\end{theorem}

\begin{remark}
   $(i)$ Theorem \ref{thm-exact-controllability} presents the exact difference between
  the reachable set of the control system (\ref{our-system}) with and without memory, and clarifies the role of the space $\mathcal H^4$.

$(ii)$ As mentioned above, Theorem \ref{thm-exact-controllability} ensures only the partial controllability of the system \eqref{our-system}. In that sense the result is weaker than the ones in    \cite{Chaves-Silva-Zhang-Zuazua} about the memory-type null controllability of the system where the memory term is also controlled.

 But in Theorem \ref{thm-exact-controllability} the control  region $Q$ is measurable, while in most related works, it is required to be open.  We refer to    \cite{Luis-wang} for
   the null controllability of the heat equation with  measurable control regions.

   $(iii)$ Theorem \ref{thm-necessity-geometry} shows that the
   MOC is nearly sharp for the partial controllability to hold.
\end{remark}

\subsection{Proofs of main results}

We start with the following technical lemma.
\begin{lemma}\label{lem-reachable-subspace-H4}
		Assume that $(Q,0,T)$ satisfies the MOC and let  $y_0\in L^2(\Omega)$.
	Then
	\begin{align}\label{H4-reachable-infty}
	\mathcal R_M^{\infty}(T,y_0,\alpha)
    =\mathcal H^4  \subset      \mathcal R_M^{\infty}(T,y_0)
    ~~~\mbox{for all}~  \alpha>1,
	\end{align}
where $\mathcal R_M^{\infty}(T,y_0,\alpha)$  is as in \eqref{202202yb-H4-WeightedReachableSet}.
\end{lemma}

\begin{proof}
We fix $\alpha>1$. Since $\varPhi(t)y_0=y(t,\cdot;y_0)$ ($t\geq 0$),
it follows from \eqref{202202yb-H4-WeightedReachableSet} and \eqref{reachable-subspace-M} that
\begin{align*}
 \mathcal R_M^{\infty}(T,0,\alpha) + \varPhi(T) y_0
 = \mathcal R_M^{\infty}(T,y_0,\alpha)
 \subset
 \mathcal R_M^{\infty}(T,y_0)
 = \mathcal R_M^{\infty}(T,0) + \varPhi(T) y_0.
\end{align*}
Accordingly, since $\varPhi(T) y_0 \in \mathcal H^4$ (see Corollary \ref{cor-only-h1}),
 it suffices to show the equality in \eqref{H4-reachable-infty}  with $y_0=0$.

   For this purpose, we write
\begin{align}\label{4.14-2-28}
    \begin{aligned}
        L^1_{\alpha} (0,T;L^2(\Omega))  :=&
        \Big\{ f:(0,T)\rightarrow L^2(\Omega)  ~\big|~
        \int_0^{T} \|f(t)\|_{L^2(\Omega)} (T-t)^{\alpha} dt <+\infty \Big\},
            \\
        L^\infty_{ -\alpha }(0,T;L^2(\Omega)):=&
        \Big\{ g:(0,T)\rightarrow L^2(\Omega)  ~\big|~
        \underset{0<t<T}{\text{ess-sup}} \,
        \|(T-t)^{-\alpha} g(t)\|_{L^2(\Omega)}
        <+\infty \Big\}.
    \end{aligned}
\end{align}
We define three Banach spaces as follows:
$
   X:=\mathcal H^{-4}, ~
   Y:=L^1_{ \alpha }(0,T;L^2(\Omega)),
   ~Z:=L^2(\Omega).
$
We can directly check that
\begin{align}\label{4.15-2-28}
   X^*=\mathcal H^{4}, ~
   Y^*=L^\infty_{- \alpha }(0,T;L^2(\Omega)),
   ~Z^*=L^2(\Omega).
\end{align}
Then, we define two operators $\mathcal R:Z\rightarrow X$ and $\mathcal O: Z
\rightarrow Y$ in the following manner:
\begin{align}\label{202202yb-Def-TwoOperators-RO}
   \mathcal R z:= z,~z\in Z
   ~~\text{and}~~
   \mathcal O z:= \chi_Q\varPhi(T-\cdot)z,~z\in Z.
\end{align}
By Proposition \ref{prop-constant-variation} (in the Appendix) and the self-adjointness of $\varPhi(\cdot)$ (see \eqref{eq-varPhi-expression} in the Appendix), we can directly check that \begin{align}\label{202202yb-Def-TwoDualOperators-RO}
   \mathcal R^* z = z,~z\in X^*
   ~~\text{and}~~
   \mathcal O^* u:= y(T;0,u),~u\in Y^*.
\end{align}
 Meanwhile, since $(Q,0,T)$ satisfies the MOC,
it follows by \eqref{Q-S-T-bounded-from-below} that  $(\widehat Q,0,T)$  also
satisfies the MOC,   where
\begin{align*}
   \widehat Q:=\big\{(t,x) \in (0,T)\times \Omega
   ~:~ (T -t,x) \in Q \big\}.
\end{align*}
 Thus, we can use    Theorem \ref{202210-thm-weighted-observability}
to see that   \eqref{202205TJYB-TwoSideObservability-with-weight} holds with $Q$ replaced by $\widehat Q$.
    This, along with  \eqref{202202yb-Def-TwoOperators-RO},
   yields that  for some $C>0$,
\begin{align}\label{202202yb-TwoOperators-RO-DoubleSidesOb}
  \frac{1}{C} \| \mathcal R z\|_{X}
  \leq \| \mathcal O z \|_{Y}
  \leq   C \| \mathcal R z\|_{X},\;\; z\in Z.
\end{align}

We now claim that
\begin{align}\label{202202yb-TwoOperators-RO-SameRanges}
   \text{Range}\, \mathcal R^* = \text{Range}\, \mathcal O^*.
\end{align}
Indeed, from the first inequality in \eqref{202202yb-TwoOperators-RO-DoubleSidesOb}, we apply Corollary \ref{202202yb-corollary-framework} to see that  for each $z\in X^*$, there is a $u\in Y^*$ such that
$
   \mathcal R^* z = \mathcal O^* u
$,
which yields  that Range\,$\mathcal R^*\subset$\,Range\,$\mathcal O^*$.
Similarly, from the second inequality in \eqref{202202yb-TwoOperators-RO-DoubleSidesOb} and Corollary \ref{202202yb-corollary-framework}, we can see that  Range\,$\mathcal R^*\supset$\,Range\,$\mathcal O^*$. Hence, \eqref{202202yb-TwoOperators-RO-SameRanges} is true.

Finally, from  \eqref{202202yb-Def-TwoDualOperators-RO},
  \eqref{202202yb-TwoOperators-RO-SameRanges},  \eqref{4.14-2-28}, and  \eqref{4.15-2-28}, we see that
\begin{align*}
  \mathcal H^4=\Big\{
   y(T;0,u) \in L^2(\Omega)
   ~:~u\in Y^*=L^\infty_{ -\alpha }(0,T;L^2(\Omega))
  \Big\}.
\end{align*}
This, along with  the definition of $\mathcal R_M^{\infty}(T,y_0,\alpha)$  (given in \eqref{202202yb-H4-WeightedReachableSet}), shows that the equality in \eqref{H4-reachable-infty}
(with  $y_0=0$) is true.
This finishes the proof of Lemma \ref{lem-reachable-subspace-H4}.
\end{proof}

We are now in a position to prove Theorem \ref{thm-exact-controllability}.

\begin{proof}[Proof of Theorem \ref{thm-exact-controllability}]
	First, we recall that $y(\cdot;y_0,u)$ denotes the solution to
\eqref{our-system}, while $z(\cdot;y_0,u)$ denotes the solution to
\eqref{our-system} where $M=0$.
	We arbitrarily fix  the initial datum $y_0\in L^2(\Omega)$. We notice that \eqref{202202yb-H4-WeightedReachableSet} directly follows from Lemma \ref{lem-reachable-subspace-H4}.
It remains  to show
\eqref{exact-ob-with-H8H6-target}, i.e.,
	\begin{align}\label{Rm=R0+H4-eq}
\mathcal H^4 \subset
	\mathcal R_M^p(T,y_0) =  \mathcal R_0^p(T,y_0) + \mathcal H^4.
	\end{align}
The proof of \eqref{Rm=R0+H4-eq} is organized in several steps.
\vskip 5pt
\noindent{\it Step 1. We show that for each $u\in L^1_{loc}(\overline{\mathbb R^+};L^2(\Omega))$,
\begin{align}\label{Rm=R0+H4-up-fu}
	f_u(\cdot)\in C( \mathbb R^+; \mathcal H^4),
	\end{align}
where
		\begin{align}\label{7.6,7.27}
	f_u(t):=
	y(t;y_0,u) -z(t;y_0,u),\;\;t>0.
	\end{align} }
~~~\,To this end, we arbitrarily fix a $u\in L^1_{loc}(\overline{\mathbb R^+};L^2(\Omega))$. Then,
	 by \eqref{7.6,7.27}, Proposition \ref{prop-constant-variation} (in the Appendix), and Corollary \ref{cor-only-parabolic},
	we find that
	\begin{align}\label{fu-u-Rc-eq-july15}
	f_u(t) = \widetilde{\mathcal R}_c(t,-A) A^{-2} y_0
	+  \int_0^t \widetilde{\mathcal R}_c(t-s,-A) A^{-2} (\chi_Qu)(s)ds,\;\;t>0.
	\end{align}
		Meanwhile,  it follows by Corollary \ref{cor-only-parabolic} that
$
	\widetilde{\mathcal R_c} \in C(\mathbb R^+;\mathcal L(\mathcal H^0)) \cap L^\infty_{loc}(\overline{\mathbb R^+};\mathcal L(\mathcal H^0)).
$
	This, along with  \eqref{fu-u-Rc-eq-july15}, yields  that when $t_2\geq t_1>0$,
	\begin{align*}
	\| f_u(t_1) - f_u(t_2) \|_{ \mathcal H^4 }
	\leq&   \| \widetilde{\mathcal R}_c(t_1,-A)
	- \widetilde{\mathcal R}_c(t_2,-A)  \|_{ \mathcal L(\mathcal H^0) }
	\| y_0\|_{\mathcal H^0}
	+ \| \widetilde{\mathcal R}_c(\cdot,-A) \|_{ L^\infty(0,t_2; \mathcal L(\mathcal H^0) ) }
	\nonumber\\
	&  \times
	\Big( \int_0^{t_1} \|u(t_1-s) - u(t_2-s)\|_{\mathcal H^0} ds
	+  \int_{t_1}^{t_2} \| u(t_2-s)\|_{\mathcal H^0} ds  \Big),
	\end{align*}
	which leads to \eqref{Rm=R0+H4-up-fu}.

\vskip 5pt

\noindent{\it Step 2. We show that
	\begin{align}\label{Rm=R0+H4-up}
	\mathcal R_M^p(T,y_0) \subset \mathcal R_0^p(T,y_0) + \mathcal H^4.
	\end{align}}
~~~\,We arbitrarily fix a $y_1 \in \mathcal R_M^p(T,y_0)$. By \eqref{reachable-subspace-M}, there is a $u_{1} \in L^p(\mathbb R^+;L^2(\Omega))$ such that
	\begin{align*}
	y_1=y(T;y_0,u_1)=z(T;y_0,u_1) +
	\Big( y(T;y_0,u_1)-z(T;y_0,u_1) \Big).
	\end{align*}
	Since $z(T;y_0,u_1) \in \mathcal R_0^p(T,y_0)$, the above, along with \eqref{Rm=R0+H4-up-fu}, leads to \eqref{Rm=R0+H4-up}.

\vskip 5pt

\noindent{\it Step 3. We show that
	\begin{align}\label{Rm=R0+H4-down}
	\mathcal R_M^p(T,y_0) \supset \mathcal R_0^p(T,y_0) + \mathcal H^4.
	\end{align}}
~~~\,We arbitrarily fix two functions $\hat y_1 \in \mathcal R_0^p(T,y_0)$ and  $\hat y_2 \in \mathcal H^4$.
According to the definition of $\mathcal R_0^p(T,y_0)$ (see \eqref{reachable-subspace-M} with $M=0$), there is a $\hat u_{1} \in L^p(\mathbb R^+;L^2(\Omega))$ such that
	\begin{align}\label{hat-y1-heat}
	\hat y_1=z(T;y_0, \hat u_1).
	\end{align}
	Since $\hat y_2 \in \mathcal H^4$, we see from \eqref{Rm=R0+H4-up-fu} that
	\begin{align}\label{def-hat-y3}
	\hat y_3: =\hat y_2 - 	\Big( y(T;y_0,\hat u_1) - z(T;y_0,\hat u_1) \Big)
	\in \mathcal H^4.
	\end{align}
Since $\hat y_3\in  \mathcal H^4$,  we can apply Lemma \ref{lem-reachable-subspace-H4} (where $y_0=0$) to find a control $\hat u_2\in L^\infty(\mathbb R^+;L^2(\Omega))$, with $\hat u_2|_{(T,+\infty)}=0$, such that
	$\hat y_3= y(T;0,\hat u_2)$.
	This, together with \eqref{def-hat-y3} and \eqref{hat-y1-heat}, yields
	\begin{align*}
	y(T;y_0,\hat u_1 + \hat u_2)
	=   y(T;y_0,\hat u_1 )  +\hat y_3
	=\hat y_2 +  z(T;y_0,\hat u_1)  =\hat y_2 + \hat y_1.
	\end{align*}
	Since $\hat y_1$ and $\hat y_2$ were arbitrarily taken
from $\mathcal R_0^p(T,y_0)$ and $\mathcal H^4$, respectively, the above leads to \eqref{Rm=R0+H4-down}.

\vskip 5pt

\noindent{\it Step 4. We check  \eqref{Rm=R0+H4-eq}.}

 Because $(Q,0,T)$ satisfies the MOC,
by Lemma \ref{lem-reachable-subspace-H4}, it follows that $\mathcal H^4 \subset \mathcal R_M^{\infty}(T,y_0)$. At the same time, $\mathcal R_M^{\infty}(T,y_0) \subset \mathcal R_M^{p}(T,y_0)$ by their definitions. Therefore, $\mathcal H^4 \subset \mathcal R_M^{p}(T,y_0)$.
This, along with  \eqref{Rm=R0+H4-up} and \eqref{Rm=R0+H4-down},  yields \eqref{Rm=R0+H4-eq}.

Hence, we have completed the proof of Theorem \ref{thm-exact-controllability}.
\end{proof}

We end this section by proving Theorem \ref{thm-necessity-geometry}.

\begin{proof}[Proof of Theorem \ref{thm-necessity-geometry}]
 Assume that the control system \eqref{our-system} is $L^p$-partially controllable over $[0,T]$.
By contradiction, we
suppose that
  \eqref{202203-MCCofBall} is not true. Then,
 there is an $r_1>0$ such that
 \begin{align*}
     \inf_{x_1\in\Omega}  {\int\hspace{-1.05em}-}_{ B(x_1,r_1) }
      \bigg( \int_0^T \chi_Q(t,x) dt
      \bigg) dx =0.
 \end{align*}
 Using  the same arguments used for \eqref{3.8-3-20} and \eqref{sharp-geometry-pf-1}, we can find
   $x_0 \in  \Omega$ and $r>0$ such that
\begin{eqnarray}\label{sharp-geometry-pf-1-R1}
\chi_Q(t,x) =0 ~~\text{for a.e.}~
(t,x) \in (0,T) \times B(x_0,r).
\end{eqnarray}

 	We now claim \color{black} that there is a constant $C>0$  such that
	for each $y_0\in L^2(\Omega)$, there exists $u_{y_0}\in L^1(0,T;L^2(\Omega))$
	satisfying
	\begin{align}\label{2020-july-null-ob-cost-R1}
	y(T;y_0,\tilde u_{y_0})=0  \;\;\mbox{and}\;\;
	\|u_{y_0} \|_{ L^1(0,T;L^2(\Omega)) }
	\leq C \|y_0\|_{L^2(\Omega))}
	\end{align}
(here and in what follows, given a control $v$ over $[0,T]$, we use $\tilde v$
to denote its zero extension  over $\mathbb R^+$).

To show \eqref{2020-july-null-ob-cost-R1}, we  define the operator
	\begin{align*}
	L_T(u):= y(T;0,\tilde u),\;\; u\in L^p(0,T;L^2(\Omega)).
	\end{align*}
	Then, by the assumption of the $L^p$-partial controllability, we have  Range\,$\varPhi(T)\subset$\,Range\,$L_T$. Without loss of generality, we can assume that $L_T$ is injective; otherwise, we can replace $L_T$ by the operator $\widetilde L_T$ (from the quotient space $L^p(0,T;L^2(\Omega))/\ker L_T$ to $L^2(\Omega)$), which is uniquely induced by $L_T$.
	Then, for each $y_0\in L^2(\Omega)$, there is a unique $u_{y_0}$ such that
	\begin{align*}
	\varPhi(T)y_0 = L_T u_{y_0},\; \mbox{i.e.},\; y(T;y_0,-\tilde u_{y_0})=0.
	\end{align*}
	According to the closed graph theorem,  we can directly check that the map $y_0\mapsto u_{y_0}$ is continuous from $L^2(\Omega)$ to the space $L^p(0,T;L^2(\Omega))$. This yields \eqref{2020-july-null-ob-cost-R1}, since $p\geq 1$.

	We next claim that there is a $C>0$ such that
		\begin{eqnarray}\label{sharp-geometry-pf-s1-2-R1}
		\|\varPhi(T)z\|_{L^2(\Omega) }
		\leq  C \| \varPhi(\cdot)z\|_{ L^\infty(0,T;L^2( \Omega\setminus B(x_0,r) )) }\;\;\mbox{for each}\;\; z\in L^2(\Omega).
		\end{eqnarray}
Indeed, by \eqref{2020-july-null-ob-cost-R1}, using the classical duality argument
	(see for instance \cite[Theorem 1.18]{WWXZ}), we can obtain the following observability inequality: there is a constant $C>0$ such that
	\begin{eqnarray}\label{5.22,7.26w-R1}
	\|\varPhi(T)z\|_{L^2(\Omega) }
	\leq
	C \|\chi_Q \varPhi(T-\cdot)z\|_{ L^{\infty}(0,T;L^2(\Omega))}\;\;\mbox{for each}\;\; z\in L^2(\Omega).
	\end{eqnarray}
	Now \eqref{sharp-geometry-pf-s1-2-R1} follows from \eqref{5.22,7.26w-R1} and  (\ref{sharp-geometry-pf-1-R1}).

Finally, we notice that \eqref{sharp-geometry-pf-s1-2-R1}  is the same as \eqref{sharp-geometry-pf-s1-2}. Thus, we can use the same arguments as those after \eqref{sharp-geometry-pf-s1-2} (in the proof of Theorem \ref{thm-null-observability-without-weight}) to get to a contradiction.
Hence, the conclusion in Theorem \ref{thm-necessity-geometry} is true. This completes the proof of Theorem \ref{thm-necessity-geometry}.
\end{proof}

\section{Numerical experiments}\label{202310-yb-SimpleExplanationForDecomposition}

The hybrid parabolic-hyperbolic effect of equation \eqref{our-obserble-system} was shown in \cite{WZZ-1}, and  plays a key role in the study of this paper.
Here we present some numerical experiments in one space dimension confirming this hybrid behaviour.

Let $\Omega=(0,1)$ and consider the following equation with the constant memory kernel:
\begin{align}\label{202320-yb-ModelWithConstantMemoryKernel}
    y'(t)  - A y(t) + \int_0^t y(s) ds =0, ~t>0;
    ~  y(0)= y_0.
\end{align}
Recall that $\{\eta_j\}_{j\geq 1}$ and $\{e_j\}_{j\geq 1}$ are the eigenvalues and the corresponding eigenvectors (normalized in $L^2(\Omega)$) of the operator $-A$, respectively. By the spectral method, we have that for each $t\geq 0$ and $x\in (0,1)$,
\begin{align*}
   y(t,x; \delta_{0.3} )
   =& \sum_{j\geq 1}   \frac{1}{2} \bigg[
   \Big( 1 + \frac{\eta_j}{  \sqrt{\eta_j^2 -4} }
   \Big)
   e^{ \big( -\eta_j - \sqrt{\eta_j^2 -4}  \big)
   \frac{t}{2} }
   \nonumber\\
   &   \quad\quad\quad
    + \Big( 1 - \frac{\eta_j}{  \sqrt{\eta_j^2 -4} }
   \Big)
   e^{ \big( -\eta_j + \sqrt{\eta_j^2 -4} \big)
   \frac{t}{2}  }
   \bigg]
   e_j(0.3) e_j(x).
\end{align*}
We discretize the interval $\Omega=(0,1)$ with the mesh size $h=10^{-3}$, keep the first $10^3$ frequency components in the last expression, and then draw the solution $y(\cdot,\cdot; \delta_{0.3} )$ and its $4$-th space derivative in black in Figure  \ref{202310-yb-SolutionsWithWaves}.

In this figure, the red  curves (in both rows) represent the solution of the pure heat equation (with the same initial datum) and its $4$-th space derivative, respectively. The blue curves (marked as ``leading wave") in the first row represent the first nontrivial term $-M(t)A^{-2} y_0$ in the wave-like component $\mathcal W_N$ of the decomposition \eqref{0423-demcomposition-eq}. The black curves represent the complete solution of equation \eqref{202320-yb-ModelWithConstantMemoryKernel} and its fourth-order derivative.
\begin{figure}[!htb]
    \centering
    \includegraphics[height=15cm,width=15cm]{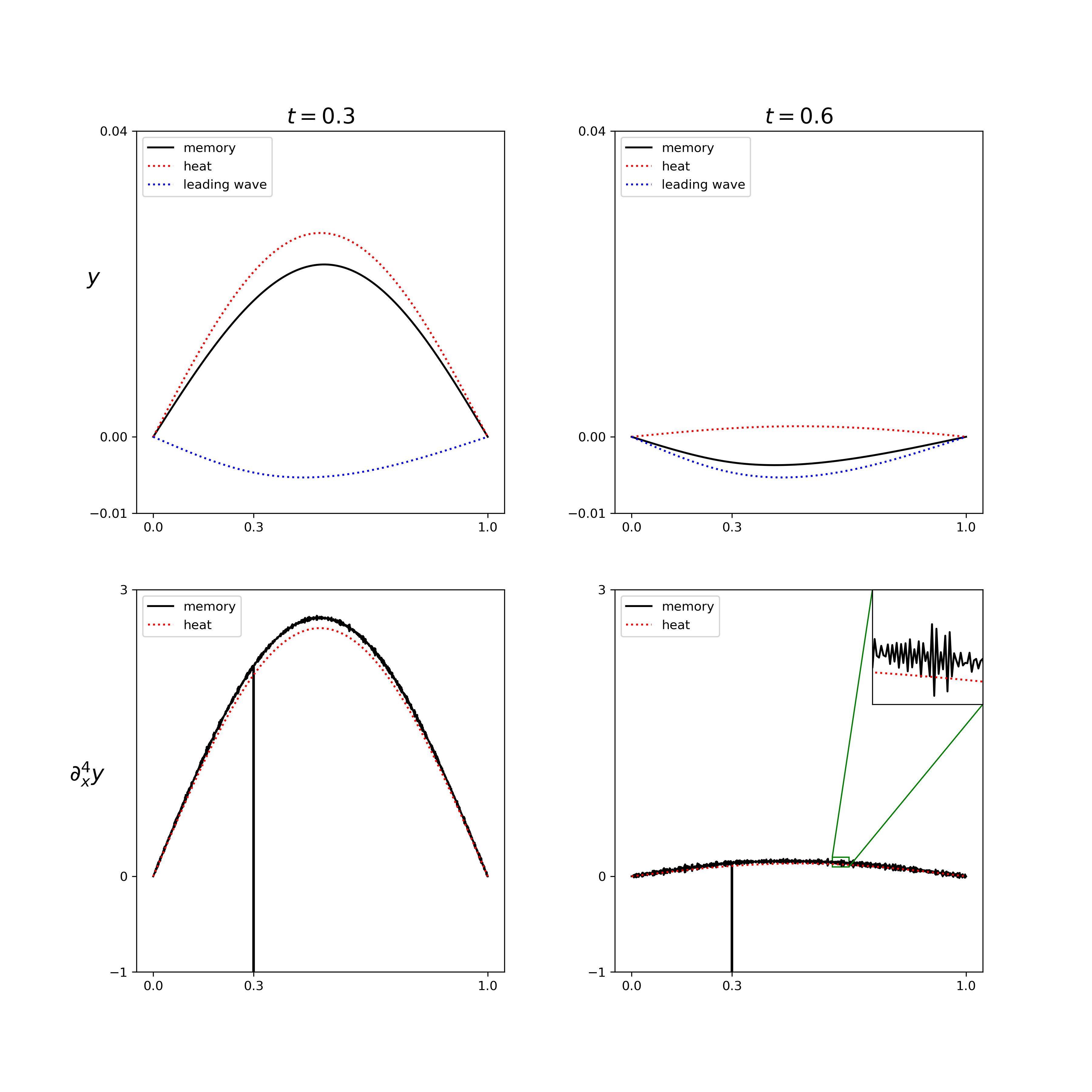}
    \caption{The curves in the first and second columns  correspond to the functions (drawn in different colors) at the time instants $t=0.3$ and $t=0.6$, respectively. The curves in the first and second rows correspond to the functions (drawn in different colors) and their $4$-th space derivatives, respectively.}
    \label{202310-yb-SolutionsWithWaves}
\end{figure}

In these figures we confirm that: $(i)$ at both two time instants, the $4$-th space derivative of $y(\cdot,\cdot; \delta_{0.3} )$ reproduces the singularity of the initial datum $\delta_{0.3}$. This confirms the propagation of singularities along time of the solutions of \eqref{our-obserble-system};  $(ii)$ the solution $y(\cdot,\cdot; \delta_{0.3} )$ evolves gradually from the pure heat solution to the vicinity of the first nonzero term $-M(t)A^{-2} y_0$
 in the wave-like component $\mathcal W_N$ in the decomposition \eqref{0423-demcomposition-eq}.

Next, we focus on the simplified equality \eqref{202310-yb-SimpliedKeyEquality}: when $T>S>0$,
\begin{align}\label{202310-yb-RewriteSimpliedKeyEquality}
    y(t,\cdot;y_0) = - M(t) A^{-2} y_0 +
    \text{``small terms"},~~t \in [S,T],
\end{align}
which is used to highlight the wave-like effect in the decomposition \eqref{0423-demcomposition-eq} for equation \eqref{our-obserble-system}.

Let us first motivate the decomposition  \eqref{202310-yb-RewriteSimpliedKeyEquality}. Equation  \eqref{202320-yb-ModelWithConstantMemoryKernel}  for the constant memory kernel, by the spectral method, leads to the following  one,  depending on the parameter $\eta>0$:
\begin{align}\label{202310-yb-ParameterizedODEsWithConstantMemory}
    x'(t)  + \eta x(t) + \int_0^t x(s) ds =0, ~t>0;
    ~  x(0)=1.
\end{align}
The solution takes the form
\begin{align*}
    x(t)= \big( 1+ o(1) \big)  \exp\big[- \big(1+o(1)\big) \eta t \big]
    - \big(  1+ o(1)  \big)   \eta^{-2} \exp\big[- \big(1+o(1)\big) \eta^{-1} t \big],
\end{align*}
where each $o(1)$ stands for  infinitesimal perturbations as $\eta\rightarrow +\infty$. In particular, when $T>S>0$,
\begin{align*}%\label{202320-yb-SimplifiedExpansionForODEs}
    x(t) = - \eta^{-2} + o(\eta^{-2}),  ~S\leq t \leq T,
\end{align*}
which is then transformed to \eqref{202310-yb-RewriteSimpliedKeyEquality} for $M\equiv 1$ by the spectral calculus (that allows to replace the parameter $\eta$  by the operator $-A$).

The same analysis can be extended to more general memory kernels, \begin{align}\label{202310-yb-ParameterizedODEsWithGeneralMemory}
    x'(t)  + \eta x(t) + \int_0^t M(t-s) x(s) ds =0, ~t>0;
    ~  x(0)=1.
\end{align}
Let $T>S>0$. We adopt the following ansatz with coefficients $\{a_k(\cdot)\}_{k=0}^2$ and $\{b_k(\cdot)\}_{k=0}^2$ to be determined:
\begin{align*}
    x(t)= e^{-\eta t} \Big( \sum_{k=0}^2 a_k(t) \eta^{-k}  + o(\eta^{-2})   \Big)
    +  \sum_{k=0}^2 b_k(t) \eta^{-k}  + o(\eta^{-2})   ,
    ~t\in [0,T],
\end{align*}
and take it into \eqref{202310-yb-ParameterizedODEsWithGeneralMemory} to get that $b_0=b_1=0$ and $b_2 = -M$. Then, we obtain 
\begin{align*}
    x(t) = - M(t) \eta^{-2} + o(\eta^{-2}),  ~t\in[S,T] \subset (0,T],
\end{align*}
to justify \eqref{202310-yb-RewriteSimpliedKeyEquality}. Theorem \ref{cor-0423-demcomposition} is a consequence of this kind of analysis.

\section{Open problems}
%%LEAP%%%\label{sec6}
\label{sec-open-problems}

A number of interesting issues could be considered in connection with the
results and methods developed in this paper. Here, we briefly present some of them.
\begin{itemize}
\item  \textit{Memory-type controllability.} We have analysed the partial controllability \eqref{partial-control}. The problem of memory-type  controllability \eqref{20231013-yb-LongTimeEquivibrillum} under the MOC for general analytic kernels, as considered here, is open.

~~~\,Note that the methods in \cite{Chaves-Silva-Zhang-Zuazua}, imposing stronger  geometric conditions on the control sets and restricting the analysis to specific kernels, like polynomial ones, in particular, allow to ensure the full control of the system \eqref{our-system}. Let us explain why those methods are insufficient to handle the more general setting in this paper.

~~~\,When, for instance, the memory kernel is polynomial, $M(t)= \sum_{j=0}^{m} a_k t^k$, a key idea in \cite{Chaves-Silva-Zhang-Zuazua} is to  rewrite  \eqref{our-system} as a hybrid control system of PDEs and ODEs which takes the form
\begin{align*}
    &\partial_t y  - \Delta y + z_0 = \chi_Q u,
        \nonumber\\
    &\partial_t z_k = z_{k+1} + M^{(k)}(0) y ,  ~k=0,\ldots,m-1,
       \\
    &\partial_t z_{m} = M^{(m)}(0) y,
        \nonumber
\end{align*}
by introducing the following extra state variables $\{z_k\}_{k=0}^m$:
\begin{align*}
    z_k(t,x):= \int_0^t M^{(k)}(t-s) y(s,x) ds, ~ t\geq 0, ~x\in \Omega,~k=0,\ldots,m.
\end{align*}
In \cite{Chaves-Silva-Zhang-Zuazua} the control $u$ is built so that the ensemble of the state variables $y,z_0,\ldots,z_m$ are controlled at the final time $T$, which guarantees the memory-type controllability of the system \eqref{our-system}. The null controllability of the above augmented system is proven by duality, as a consequence of an observability inequality from a moving observation set, employing Carleman inequalities.  This requires however some stronger MOC conditions.

~~~\,The methods in \cite{Chaves-Silva-Zhang-Zuazua} seem insufficient to deal both with the sharp MOC condition in this paper and general analytic kernels. Note in particular that, in the case of general analytic memory kernels, the strategy above of adding auxiliary state variables  leads to the coupling of a heat equation with an infinite number of ODEs, which makes it hard to implement methods inspired by Carleman inequalities.
More precisely, the  memory-type controllability \eqref{20231013-yb-LongTimeEquivibrillum}  is open in the following three cases:
\begin{itemize}[leftmargin=4em]
    \item[$(i)$]  under the stronger moving geometric conditions as in \cite{Chaves-Silva-Zhang-Zuazua} for general analytic memory kernels;

    \item[$(ii)$] under the sharp MOC condition of the present paper for polynomial memory kernels;

    \item[$(iii)$] under the sharp MOC condition of the present paper for general analytic memory kernels.
\end{itemize}

    \item \textit{Smooth memory kernels.} It would be interesting to investigate whether Theorem \ref{202205TJyb-MainTheorem-merged} holds when
    $M\in C^{\infty}([0,+\infty))$.
Extending the method we developed for the real analytic memory kernels to this $C^\infty$-case, we need the following two results: first, an analog of the decomposition in Theorem \ref{cor-0423-demcomposition}; second, a unique continuation theorem similar to Lemma \ref{lem-unique-continuation-for-low-frequence} (there we used the analyticity in the time variable of the solutions to equation \eqref{our-obserble-system}, presented in Proposition \ref{pro-analyticity-pointwise}). The first result holds for smooth memory kernels (see \cite[Theorem 4.10]{WZZ-1} for the details). However, without assuming the analyticity of the memory kernel, we may not have the time-analyticity of solutions, and then the unique continuation property is a challenging open problem.

    \item \textit{Space-dependent memory kernels.} The extension of the results
    of this paper to the space-dependent memory kernels $M=M(t,x)$ is also open.

    ~~~\,In particular,  we do not know how to reveal the hybrid parabolic-hyperbolic property, with a decomposition similar to Theorem \ref{cor-0423-demcomposition}. The unique continuation property requires further analysis as well.

    \item \textit{Memory kernels in the principal part of the model.} It would
    be interesting to extend Theorem \ref{202205TJyb-MainTheorem-merged} to the following
    two types of heat equations with memory kernels:
    \begin{itemize}
        \item[($i$)]
        $\partial _{t} y - \Delta y - \int _{0}^{t} M(t-s) \Delta y(s)ds=0$;
        \item[($ii$)]
        $\partial _{t} y - \int _{0}^{t} M(t-s) \Delta y(s)ds=0$.
    \end{itemize}
These models are more relevant  than \eqref{our-obserble-system} from an applied perspective. The memory kernel enters within the principal part of equations, and this may bring new phenomena. It would be interesting to establish a similar decomposition as in Theorem \ref{cor-0423-demcomposition} to reveal the possible hybrid parabolic-hyperbolic character of these models and its control theoretical consequences.

\end{itemize}

\section{Appendix}\label{sec-appendix}

This section  reviews the decomposition theorem for the flow $\varPhi(t)$ and some
other properties established in \cite{WZZ-1}.
%We also present some numerical experiments that confirm this behaviour.

 Next, an estimate for the heat equation, as well as the proofs of several results developed before,
 is provided. Then, the variation of the constant formula for the control system \eqref{our-system} is presented.  Finally, an abstract framework for observability/controllability (in \cite{WWZ-jems}) is introduced.

\subsection{Review of decomposition of flow}

 We start by recalling  the definition of the following functions from \cite{WZZ-1}.
First,  for each  $l\in\mathbb{N}$, we let
 \begin{align}\label{thm-ODE-meomery-asymptotic-estimate-hypobolic}
\begin{cases}
h_l(t) :=&
(-1)^l \displaystyle\sum_{j=0}^l  C_{l}^{l-j}
\dfrac{ d^{(l-j)} }{ dt^{(l-j)} } \underset{j}{ \underbrace{M*\cdots*M} } (t),\; t\geq 0,
\\
p_l(t) :=& -h_l(0)+(-1)^{l+1}
\displaystyle\sum_{  \tiny\begin{array}{c}
	m,j\in \mathbb N^+, \\
	2j-l-1\leq m \leq j
	\end{array}
}
\bigg(
C_{l}^{l-j+m}
\dfrac{ d^{(l-j+m)} }{ dt^{(l-j+m)} } \underset{j}{ \underbrace{M*\cdots*M} }(0)
\bigg)  \frac{(-t)^m}{m!},~t\geq 0,
\end{cases}
\end{align}
where $C_\beta^m:= \frac{\beta!}{m!(\beta-m)!}$ and
  $\underset{j}{ \underbrace{M*\cdots*M} }:=0$ if $j=0$. Second, we let
  \begin{align}\label{new-kernel-KM}
	K_M(t,s):=
	\sum_{j=1}^{+\infty}  \frac{ (-s)^j }{ j! } \underset{j}{ \underbrace{M*\cdots*M} }(t-s),~~t\geq s.
	\end{align}
The next decomposition theorem and Proposition \ref{prop-varPhi-expression} below are consequences of
  \cite[Theorems 1.1 and 1.2]{WZZ-1} and \cite[Propositions 2.3 and 4.8]{WZZ-1}, respectively.

\begin{theorem}\label{cor-0423-demcomposition}
	 For each integer $N\geq 2$,  it holds that
\begin{align}\label{0423-demcomposition-eq}
	\varPhi(t) =&  \mathcal P_N(t)+  \mathcal W_N(t)  +  \mathfrak R_N(t),~~t\geq0,
	\end{align}
with
 \begin{align}\label{def-PN-HN-RN}
\left\{
	\begin{array}{lll}
	  \mathcal P_N(t)&:=& e^{tA} + e^{tA}\sum_{l=0}^{N-1} p_{l}(t)  (-A)^{-l-1}   ,\\
	  \mathcal W_N(t)&:=& \sum_{l=1}^{N-1} h_{l}(t)  (-A)^{-l-1},\\
	  \mathfrak R_N(t)&:=& R_N(t,-A) (-A)^{-N-1},
	\end{array}
~~~~~	t\geq0,
\right.
\end{align}
 where $p_l$ and $h_l$ are given by \eqref{thm-ODE-meomery-asymptotic-estimate-hypobolic}
and
 \begin{align}\label{0921-RN-good-remainder}
	 R_N(t,\tau) := \int_0^t \tau e^{-\tau s} \partial_s^N K_M(t,s)  ds,~~
	t\geq0,~~\tau\geq 0,
	\end{align}
with $K_M$ given by \eqref{new-kernel-KM}.
  Moreover, for  each $t\geq 0$,   $\{h_l(t)\}_{l\geq1}$
	is not the null sequence, while   for each  integer $N\geq 2$ and
   each $s\in \mathbb R$,
    $R_N(\cdot,-A)|_{\mathbb R^+}$ belongs to $C(\mathbb R^+;\mathcal L( \mathcal H^s))$ and satisfies
        \begin{eqnarray}\label{RN-property-regularity}
	\big\|  R_N(t,-A) \|_{ \mathcal L( \mathcal H^s) }
	\leq e^t \bigg\{
	\exp\bigg[ N(1+t) \bigg(\sum_{j=0}^N \max_{0\leq s\leq t} \Big| \frac{d^j}{ds^j}M(s) \Big| \bigg) \bigg]  - 1  \bigg\}
	,~~t\geq 0.
	\end{eqnarray}
 	\end{theorem}

 \begin{remark}\label{202210-TJyb-remark-WN}
     When $N\geq 3$, the wave-like component $\mathcal W_N$ in \eqref{def-PN-HN-RN} is of the form :
     \begin{align}\label{202210-TJyb-WN-LeadingTerm}
         \mathcal W_N(t) = -M(t) A^{-2} +   \sum_{l=2}^{N-1} h_{l}(t)  (-A)^{-l-1},
         ~~t\geq 0.
     \end{align}
 This can be directly checked from \eqref{def-PN-HN-RN} and \eqref{thm-ODE-meomery-asymptotic-estimate-hypobolic} (where  $h_1=-M$).
 \end{remark}

\begin{proposition}\label{prop-varPhi-expression}
Let 	$K_M$ be given by (\ref{new-kernel-KM}). Then, $K_M$ is real analytic over $S_+:=\{(t,s)\in \mathbb R^2~:~t\geq s\}$. Moreover,
\begin{align}\label{eq-varPhi-expression}
	\varPhi(t)^*=\varPhi(t)
	= e^{tA} +   \int_0^t K_M(t,\tau) e^{\tau A} d\tau,
	\;\;  t\geq 0.
	\end{align}
	\end{proposition}

\subsection{Estimate for the heat equation}

The following technical result provides an estimate for the heat equation that we present for the sake of completeness.

\begin{lemma}\label{lem-heat-up-to-0}
	Let $B(x_0,r)\subset\Omega$ be a closed ball  centered at $x_0\in\Omega$ with radius $r>0$.
	Then for each $s\in \mathbb R$, there exists a constant $C=C(s)>0$ such that
	\begin{eqnarray}\label{heat-local-estimate-Hs-L2}
	\sup_{t\geq 0} \| e^{tA}z \|_{ L^2( \Omega \setminus B(x_0,r) ) }
	\leq C  \|z\|_{\mathcal H^s}\;\;\mbox{for each}\;\;  z\in \mathcal H^s \big( B(x_0,r/2) \big),
	\end{eqnarray}
	where $\mathcal H^s \big( B(x_0,r/2) \big)  :=  \big\{  z\in \mathcal H^s~:~  \mbox{supp}\,z\subset B(x_0,r/2) \big\}$.
	\end{lemma}

\begin{proof}
	It suffices to prove (\ref{heat-local-estimate-Hs-L2}) for each $s=-m$ with $m\in \mathbb N^+$. For this purpose,
	we 	arbitrarily fix a 
	$z\in \mathcal H^{-m} \big( B(x_0,r/2) \big)$.
		We select $\{r_l\}_{l=1}^{2m} \subset \mathbb R^+$ such that
	\begin{eqnarray}\label{classic-heat-local-july-1}
	r/2<r_1<\cdots<r_{2m}<r.
	\end{eqnarray}
	Then, we take a sequence of  functions $\{\rho_l\}_{l=1}^{2m}
	\subset C_0^\infty(\mathbb R^n)$ such that for each $l\in\{1,\cdots,2m\}$,
	\begin{eqnarray}\label{classic-heat-local-july-2}
	\rho_l=0 \;\;\mbox{over}\;\; B(x_0,r_{l-1})
	\;\;\mbox{and}\;\;
	\rho_l=1 \;\;\mbox{over}\;\; \mathbb R^n\setminus B(x_0,r_{l}).
	\end{eqnarray}
	We define a sequence of functions  $\{f_l\}_{l=0}^{2m}$ in the following manner:
	\begin{eqnarray}\label{classic-heat-local-july-3}
	f_0(t):= e^{tA} z,~t>0;~~
	f_l(t) := \rho_l e^{tA} z ,
	~t>0,~l\in\{1,\cdots,2m\}.
	\end{eqnarray}
	The rest of the  proof  is organized in several steps.

	\vskip 5pt
	\noindent\textit{Step 1. We prove that for each $\chi_1\in C_0^\infty(\Omega)$, $\chi_2\in C_0^\infty(\Omega;\mathbb R^n)$, and $\alpha\geq 0$, there is a $C=C(\chi_1,\chi_2,\alpha)>0$ such that
		\begin{eqnarray}\label{classic-heat-local-july-5}
		\| \chi_1  g\|_{\mathcal H^{-\alpha-1}} +
		\| \chi_2 \cdot \nabla g\|_{\mathcal H^{-\alpha-1}}
		\leq C \| g\|_{\mathcal H^{-\alpha}},\;\mbox{when}\; g \in \mathcal H^{-\alpha}.
		\end{eqnarray}}
 $\;\;$We arbitrarily fix $\chi_1\in C_0^\infty(\Omega)$ and $\chi_2\in C_0^\infty(\Omega;\mathbb R^n)$. We claim that for each $\alpha\geq0$,  there is a $C_1=C_1(\chi_1,\chi_2,\alpha)>0$ such that
	\begin{align}\label{div-rho-f-bound}
	\| \chi_1  f\|_{\mathcal H^{\alpha}}
	+	\| \mbox{div}\, (\chi_2 f) \|_{ \mathcal H^\alpha }
	\leq C_1 \|f\|_{\mathcal H^{\alpha+1} },\;\mbox{when}\;f\in \mathcal H^{\alpha+1}.
	\end{align}
	When \eqref{div-rho-f-bound}  is proven, \eqref{classic-heat-local-july-5} follows by the standard duality argument.

By  the interpolation theorem in \cite[Theorem 5.1]{Lions-Magenes}, we see that in order to  show \eqref{div-rho-f-bound}, it suffices to prove it
for $\alpha=2k$ with $k\in \mathbb N$.  To this end, we arbitrarily fix $\alpha =2k$ (with $k\in \mathbb N$) and  $f\in  \mathcal H^{2k+1}$.
	Since $\chi_1$ and $\chi_2$ are compactly supported in $\Omega$, we have
	\begin{align}\label{August14-Omega1-Omega}
	\Omega_1 := \text{supp}\,\chi_1 \cup \text{supp}\, \chi_2
	\subset\subset  \Omega.
	\end{align}
	We claim that there is a $C_2>0$ (independent of $f$) such that
	\begin{align}\label{f-2k+1-4k+2}
	\| f\|_{H^{2k+1}(\Omega_1)}  \leq C_2 \|f\|_{\mathcal H^{2k+1}}.
	\end{align}
	In fact, given $h\in \mathcal H^{2k+1}$, we have that $A^{k} h\in \mathcal H^1$. From this, \eqref{selfadjoit-elliptic-operator}, and \eqref{def-space-with-boundary-condition}, we see that
$
	\Delta^{k} h \in \mathcal H^1=H_0^1(\Omega).
$
	Since  $\Delta^{k}$ is an elliptic operator of  order $2k$, the above shows that
 $h \in H^{2k+1}_{loc}(\Omega)$
 (see for instance \cite[Theorem 18.1.29]{Hormander-3}). Consequently, we have
 $h|_{\Omega_1} \in H^{2k+1}(\Omega_1)$.
 Thus,  we can  define a linear map $\mathcal T$ from $\mathcal H^1$ to $H^{2k+1}(\Omega_1)$
in the following manner:
	\begin{align}\label{T-map-L2-H4k+2}
	\mathcal T ( A^{k} \bar h )
	:= \bar h|_{\Omega_1},\;\;\bar h\in \mathcal H^{2k+1}.
	\end{align}
	By using   the closed graph theorem to $\mathcal T$, we deduce that it is bounded.
	Then, by \eqref{T-map-L2-H4k+2}, there is a $C_3>0$ such that
	\begin{align*}
	\| \bar h \|_{ H^{2k+1}(\Omega_1) }
	\leq C_3 \| A^{k } \bar h \|_{ \mathcal H^1 }
	= C_3 \| \bar h \|_{ \mathcal H^{2k+1} }\;\;\mbox{for each}\;\;\bar h\in \mathcal H^{2k+1},
	\end{align*}
	which leads to  \eqref{f-2k+1-4k+2}.

	Now, by \eqref{def-space-with-boundary-condition}, \eqref{selfadjoit-elliptic-operator},   and \eqref{August14-Omega1-Omega}, there is a $C_4>0$ (independent of $f$)
	such that
	\begin{align*}
	\| \chi_1  f\|_{\mathcal H^{2k}}
	+	\| \mbox{div}\, (\chi_2 f) \|_{ \mathcal H^{2k} }
	= \| \Delta^k (\chi_1  f)\|_{ L^2(\Omega) }
	+	\| \Delta^k \mbox{div}\, (\chi_2 f) \|_{ L^2(\Omega) }
	\leq C_4 \| f \|_{ H^{2k+1}(\Omega_1) }.
	\end{align*}
	The above, along with \eqref{f-2k+1-4k+2}, yields \eqref{div-rho-f-bound} with $\alpha=2k$. This ends the proof of Step 1.

	\vskip 7pt
	\noindent {\it Step 2. We  prove  that for each $l\in\{1,\cdots,2m\}$, there exists a $C_l>0$ (independent of $z$) such that
		\begin{eqnarray}\label{classic-heat-local-july-4}
		\| f_l\|_{L^\infty(\mathbb R^+;\mathcal H^{l/2-m})}
		\leq C_l  \| f_{l-1}\|_{L^\infty(\mathbb R^+;\mathcal H^{(l-1)/2-m})}.
		\end{eqnarray}}
	We will show that (\ref{classic-heat-local-july-4}) is satisfied by  induction. To this end, we first show that (\ref{classic-heat-local-july-4}) is satisfied with $l=1$. Indeed, from (\ref{classic-heat-local-july-3}), we know that
	$ f_1=\rho_1 f_0$. This, along with   (\ref{selfadjoit-elliptic-operator}), yields
	\begin{eqnarray}\label{classic-heat-local-july-6}
	\frac{d}{dt} f_1(t) - A f_1(t) = F_1(t),~t>0;
	~ f_1(0)=0,
	\end{eqnarray}
	where
	\begin{eqnarray}\label{classic-heat-local-july-7}
	F_1(t) := (- \Delta \rho_1) f_0
	-  2 \nabla \rho_1 \cdot \nabla f_0,
	~t>0.
	\end{eqnarray}
	Meanwhile, it follows from (\ref{classic-heat-local-july-2})
		that
$
		\Delta\rho_1  \in C_0^\infty(\Omega)
	$ and $
		\nabla \rho_1 \in C_0^\infty(\Omega;\mathbb R^n).
$
	From these and (\ref{classic-heat-local-july-7}),   we can apply (\ref{classic-heat-local-july-5}) (with $(\chi_1,\chi_2,\alpha)=(\Delta \rho_1,\nabla \rho_1,m)$)  to find  a constant $\widehat C_1>0$ (independent of $z$) such that
	\begin{eqnarray}\label{classic-heat-local-july-8}
	\| F_1\|_{L^\infty(\mathbb R^+;\mathcal H^{-1-m})}
	\leq \widehat C_1  \| f_0\|_{L^\infty(\mathbb R^+;\mathcal H^{-m})}.
	\end{eqnarray}
	We now claim that there is a $\widehat C_2>0$ (independent of $z$) such that
	\begin{eqnarray}\label{classic-heat-local-july-9}
	\| f_1\|_{L^\infty(\mathbb R^+;\mathcal H^{1/2-m})}
	\leq \widehat C_2  \| f_0\|_{L^\infty(\mathbb R^+;\mathcal H^{-m})}.
	\end{eqnarray}
	Indeed, from (\ref{classic-heat-local-july-6}), we can find a $C>0$ (independent of $z$) such that for each $t>0$,
	\begin{align*}
	\|f_1(t)\|_{ \mathcal H^{ 1/2-m} }
	\leq& \int_0^t \| e^{(t-s)A} \|_{\mathcal L(\mathcal H^{-1-m}; \mathcal H^{1/2-m} )}
	\|F_1(s)\|_{\mathcal H^{-1-m}} ds
	\nonumber\\
	\leq &  \bigg(	\int_0^t \Big\| \Big[ (-A)^{3/4} e^{(t-s)A/2} \Big] e^{(t-s)A/2} \Big\|_{\mathcal L(\mathcal H^{-1-m} )} ds  \bigg)
	\|F_1\|_{ L^\infty(\mathbb R^+;\mathcal H^{-1-m} )}
	\nonumber\\
	\leq& \bigg( \int_0^{t} C (t-s)^{-3/4} e^{- (t-s)\eta_1/2} ds \bigg)
	\|F_1\|_{ L^\infty(\mathbb R^+;\mathcal H^{-1-m} )}.
	\end{align*}
	This, along with \eqref{classic-heat-local-july-8},  leads to \eqref{classic-heat-local-july-9}. Thus, (\ref{classic-heat-local-july-4}) holds for $l=1$.

	Next, we assume that for some $l_0\in \{1,\cdots,2m-1\}$,  (\ref{classic-heat-local-july-4}) holds for all $l\leq l_0$.
	We aim to prove (\ref{classic-heat-local-july-4}) with $l=l_0+1$.
	In fact, from (\ref{classic-heat-local-july-3}) and (\ref{classic-heat-local-july-2}), we have
	$ f_{l_0+1}=\rho_{l_0+1} f_{l_0}$. By this and using a similar method to that used in the proof
	of  (\ref{classic-heat-local-july-4}) with $l=1$, we can determine (\ref{classic-heat-local-july-4})
	with $l=l_0+1$. This ends the proof of  Step 2.

	\vskip 5pt
	\noindent {\it Step 3. We  verify (\ref{heat-local-estimate-Hs-L2}).}
	\vskip 5pt
	\noindent Since $z\in \mathcal H^{-m}$, it follows from (\ref{classic-heat-local-july-3}) and (\ref{def-Hs-space-norm}) that
	\begin{eqnarray*}
		\|f_0(t)\|_{\mathcal H^{-m} }
		= \| e^{tA} z \|_{\mathcal H^{-m} }
		\leq  \|z\|_{\mathcal H^{-m} },\; t\geq 0.
	\end{eqnarray*}
	This, together with (\ref{classic-heat-local-july-4}) and (\ref{classic-heat-local-july-1})--\eqref{classic-heat-local-july-3}, yields (\ref{heat-local-estimate-Hs-L2}) with $s=-m$.

	Hence, we have completed the proof of Lemma \ref{lem-heat-up-to-0}.
\end{proof}

\subsection{Technical  proofs}\label{Appendix-SeveralProofs}

In this subsection, we present the proofs of several results stated before.  We start with the proof of Corollary \ref{cor-only-h1}.

\begin{proof}[Proof of Corollary \ref{cor-only-h1}]
    We arbitrarily fix a $\beta \in [2,3]$.
    Let $\mathcal R_c$ be given by Proposition \ref{cor-heat-wave-N=2-by-yb-20211126}. We define
    \begin{align*}
        \widehat{\mathcal R}_c(t,\tau) := e^{-t\tau}\Big(1 + M(0) t\tau^{-1} + M(0) \tau^{-2} \Big)\tau^{\beta}  - \mathcal R_c(t,\tau) \tau^{\beta-3},\;\;t>0,\;\tau>0.
    \end{align*}
    Then, by spectral functional calculus, we have
    \begin{align}\label{Rc-for-W-by-yb-202110}
        \widehat{\mathcal R}_c(t,-A) = e^{tA}\Big(1- M(0) tA^{-1} + M(0) A^{-2} \Big)(-A)^{\beta}  - \mathcal R_c(t,-A) (-A)^{\beta-3},\;\;t>0.
    \end{align}
    Now, \eqref{demcomposition-N=2-only-h1} follows from \eqref{demcomposition-N=2-PW} and \eqref{Rc-for-W-by-yb-202110}.

    Next, because $2\leq \beta\leq 3$,  for each $j\in\{0,1,2\}$ and each $s\in \mathbb R$,
    \begin{align}\label{2.12,2-23}
        \| (-A)^{\beta-j} e^{tA} \|_{\mathcal L( \mathcal H^s )}
        \leq \sup_{\lambda>0} \lambda^{\beta-j} e^{-t\lambda }
        \leq 2 t^{j-\beta},~~
        t>0.
    \end{align}
    Hence, \eqref{R1R2-regularity-only-h1} follows from \eqref{Rc-for-W-by-yb-202110}, \eqref{2.12,2-23}, and \eqref{R1R2-regularity-by-yb-20211126} directly. This
    completes the proof.
\end{proof}

    The proof of Corollary \ref{cor-only-parabolic} is  as follows.

\begin{proof}[Proof of Corollary \ref{cor-only-parabolic}]
    Let $\mathcal R_c$ be given by Proposition \ref{cor-heat-wave-N=2-by-yb-20211126}.
    We define
    \begin{align*}%\label{sec4-last-august-2021}
        \widetilde{\mathcal R}_c(t,\tau) :=
        e^{-t\tau}\big( tM(0)\tau + M(0) \big) - M(t)
        - \mathcal R_c(t,\tau) \tau^{-1},
        ~t>0,~\tau>0.
    \end{align*}
    Then, by spectral functional calculus, we find
    \begin{align*}
        \widetilde{\mathcal R}_c(t,-A) =
        e^{t A}\big( -tM(0) A + M(0) \big) - M(t)
        + \mathcal R_c(t,-A) A^{-1},
        ~t>0.
    \end{align*}
    This, along with \eqref{demcomposition-N=2-PW}, yields 
    \eqref{demcomposition-N=2-only-parabolic}.

    Next, 	we can directly check that   for each $s\in \mathbb R$,
    $\| tA e^{tA} \|_{\mathcal L( \mathcal H^s )}
    \leq 1$ when $t\geq 0$. This, along with \eqref{R1R2-regularity-by-yb-20211126},
    yields the desired estimate of the above $\widetilde{\mathcal{R}}_c$ and  completes the proof.
\end{proof}

  The proof of Proposition \ref{202202yb-prop-DoubleSidesOb-NecessaryWeights} is presented as follows.

\begin{proof}[Proof of Proposition \ref{202202yb-prop-DoubleSidesOb-NecessaryWeights}]
    Without  loss of generality, we can assume that $\alpha\geq 0$. Otherwise, we can replace
    $\alpha<0$ by $\alpha=0$.
    %    We obtain from \eqref{2.44-2-25} that
    %    \begin{align}\label{2.45-2-25}
        %        \sup_{y_0\in L^2(\Omega), \|y_0\|_{\mathcal H^{-4}} \leq 1}
        %        \|\varPhi(t)y_0\|_{L^1_{\mu}(0,\varepsilon_0;L^2(B(x_0,r)))}
        %        <+\infty.
        %    \end{align}
    It follows from  Corollary \ref{cor-only-parabolic} that
    the map $t \mapsto \varPhi(t)-e^{tA}$ ($t>0$) belongs to $L^\infty(\mathbb R^+; \mathcal L(\mathcal H^{-4},L^2(\Omega)))$. Meanwhile, note that $\varPhi(t)y_0=y(t,\cdot;y_0)$, $t\geq 0$. Then, by  the triangle inequality and \eqref{202210TJYB-TwoSideObservability-necessary-weight},
    we determine that
    \begin{align*}%\label{20221020-TjYb-NecessaryWeight-toPureHeat}
        &  \sup_{y_0\in L^2(\Omega), \|y_0\|_{\mathcal H^{-4}} \leq 1}
        \int_0^{T}  \|\chi_Q e^{tA}y_0\|_{L^2(\Omega)}  t^{\alpha} dt
        \nonumber\\
        \leq& \sup_{y_0\in L^2(\Omega), \|y_0\|_{\mathcal H^{-4}} \leq 1}
        \int_0^{T}  \|\chi_Q y(t,\cdot;y_0) \|_{L^2(\Omega)}  t^{\alpha} dt
        \nonumber\\
        &  \quad\quad\quad    + \sup_{y_0\in L^2(\Omega), \|y_0\|_{\mathcal H^{-4}} \leq 1}
        \int_0^{T}  \|\chi_Q \big(\varPhi(t)-e^{tA}\big) y_0\|_{L^2(\Omega)}  t^{\alpha} dt
        <+\infty.
    \end{align*}
    Since $Q\supset (0,\varepsilon_0)\times\omega$ and
    $\omega$ contains a closed ball $B(x_0,r)$ (centered at $x_0$ with radius $r$),
    from the above inequality and Lemma \ref{lem-heat-up-to-0} in the Appendix, we see that
    \begin{align*}
        \sup_{y_0\in C_0^\infty(B(x_0, r/2 )), \,\|y_0\|_{\mathcal H^{-4}} \leq 1}
        \int_0^{\varepsilon_0} \|e^{tA}y_0\|_{ L^2(\Omega)}  t^{\alpha} dt
        <+\infty,
    \end{align*}
    which shows that for some $C_1>0$,
    \begin{align}\label{upper-bound-Lp-heat-2202}
        \int_0^{\varepsilon_0} \|e^{tA}A^2z\|_{L^2(\Omega)}  t^{\alpha} dt
        \leq C_1 \|z\|_{L^2(\Omega)}\;\;\mbox{for all}\;\;z\in C_0^\infty(B(x_0, r/2 )).
    \end{align}
    Next, we take a sequence $\{y_k\}_{k=1}^{+\infty} \subset C_0^{\infty}(B(x_0, r/2 ))$ such that
    \begin{align}\label{yk-L2NormOne-WeakLimZero}
        \|y_k\|_{L^2(\Omega)} =1\;\mbox{for all}\;\;k\in \mathbb N^+;\;\;
        y_k   \;\rightharpoonup\;   0
        ~\text{weakly in}~  L^2(\Omega).
    \end{align}
    We are going to show that there is a subsequence $\{k_l\}_{l=1}^{+\infty}\subset \mathbb N^+$
    (with $k_1<k_2<\cdots$) such that
    \begin{align}\label{etA-zm-Bounds-202202}
        \limsup_{m\rightarrow+\infty}  \sum_{l=m}^{m+m_0-1}
        \int_{0}^{+\infty}
        \| e^{tA} A^2 y_{k_l}\|_{L^2(\Omega)}  t^{\alpha} dt
        \leq C_1 \sqrt{m_0}\;\;\mbox{for each}\;\;m_0\in \mathbb N^+\setminus\{1\}.
    \end{align}
    The proof of \eqref{etA-zm-Bounds-202202} is organized in the following two steps.

    \noindent{\it Step 1. We claim that there is a subsequence $\{k_l\}_{l=1}^{+\infty}\subset \mathbb N^+$ and a decreasing sequence $\{t_l\}_{l=1}^{+\infty} \subset (0,\varepsilon_0]$
        such that for each $l\in \mathbb{N}^+$,
        \begin{align}\label{ykl-almost-OrthogonalAndConcentated-2202}
            \sum_{m<l} | \langle y_{k_l}, y_{k_m} \rangle_{L^2(\Omega)} |
            +      \int_{\mathbb R^+\setminus(t_{l+1}, t_l)}
            \| e^{tA} A^2 y_{k_l}\|_{L^2(\Omega)}  t^{\alpha} dt
            \leq 2^{-l}.
        \end{align}
        Here,  $\sum_{m<l} | \langle y_{k_l}, y_{k_m} \rangle_{L^2(\Omega)}|:=0$ when $l=1$.}

    \vskip 5pt
    For this purpose, we first take $t_1=\varepsilon_0$. We next take $k_1\in \mathbb{N}^+$ sufficiently large
    that
    \begin{eqnarray}\label{2.50--2-27}
        \int_{t_1}^{+\infty}\|e^{tA}A^2y_{k_1}\|_{L^2(\Omega)}  t^{\alpha} dt\leq \frac{1}{4}
    \end{eqnarray}
    (the existence of such $k_1$ is ensured by the exponential decay property of
    $e^{tA}$ and the weak convergence in \eqref{yk-L2NormOne-WeakLimZero}).
    We then take
    $t_2\in(0,t_1)$ sufficiently small that
    \begin{eqnarray}\label{2.50--2-28}
        \int_{0}^{t_2}\|e^{tA}A^2y_{k_1}\|_{L^2(\Omega)}  t^{\alpha} dt\leq \frac{1}{4}
    \end{eqnarray}
    (the existence of such a $t_2$
    can  be obtained directly from \eqref{upper-bound-Lp-heat-2202}).
    Thus, it follows from \eqref{2.50--2-27} and \eqref{2.50--2-28}
    that  the above $k_1$ and $\{t_l\}_{l=1}^2$ satisfy \eqref{ykl-almost-OrthogonalAndConcentated-2202}
    with $l=1$.

    Next, we suppose inductively that for some $N\in \mathbb{N}^+$,
    there are $k_1< \cdots< k_N$ and $t_1>t_2>\cdots>t_N>t_{N+1}$
    satisfying  \eqref{ykl-almost-OrthogonalAndConcentated-2202}
    with $l=N$.
    We aim to find  $k_{N+1}>k_N$ and $t_{N+2}<t_{N+1}$ such that
    $\{k_l\}_{l=1}^{N+1}$ and $\{t_l\}_{l=1}^{N+2}$ satisfy \eqref{ykl-almost-OrthogonalAndConcentated-2202}
    with $l=N+1$. To this end, we use
    the weak convergence in \eqref{yk-L2NormOne-WeakLimZero} and the exponential decay
    property of $e^{tA}$ to find $\hat k>k_N$ such that
    \begin{align}\label{ykl-almost-OrthogonalAndConcentated-2202-N}
        \sum_{m<N+1} | \langle y_{\hat k}, y_{k_m} \rangle_{L^2(\Omega)} |
        +
        \int^{+\infty}_{ t_{N+1}}
        \| e^{tA} A^2 y_{\hat k}\|_{L^2(\Omega)}  t^{\alpha} dt
        \leq 2^{-1} 2^{-N-1} .
    \end{align}
    Meanwhile, since $y_{\hat k}\in C_0^\infty(\Omega)$, we have
    $\int^{\varepsilon_0}_0
    \| e^{tA} A^2 y_{\hat k}\|_{L^2(\Omega)}  t^{\alpha} dt<+\infty$.
    Thus, there is a $\hat t\in (0,t_{N+1})$ sufficiently small that
    \begin{align}\label{2.51-2-25}
        \int_0^{\hat t}
        \| e^{tA} A^2 y_{\hat k}\|_{L^2(\Omega)}  t^{\alpha} dt
        \leq 2^{-1} 2^{-N-1}.
    \end{align}
    Let $k_{N+1}:=\hat k$ and $t_{N+2}:=\hat t$.
    Then, it follows from \eqref{2.51-2-25} and  \eqref{ykl-almost-OrthogonalAndConcentated-2202-N}
    that   $\{k_l\}_{l=1}^{N+1}$ and $\{t_l\}_{l=1}^{N+2}$ satisfy
    \eqref{ykl-almost-OrthogonalAndConcentated-2202} with  $l= N+1$.
    Thus, by the induction, the claim in {\it Step 1} is true.

    \vskip 5pt
    \noindent {\it Step 2. We show that \eqref{etA-zm-Bounds-202202} is satisfied.}

    We arbitrarily fix $m_0\in \mathbb N^+\setminus\{1\}$ and  $m\in \mathbb N^+$. We define
    \begin{align}\label{zm-Jm-Defs-202202}
        z_m:=  \sum_{l\in \mathcal J_{m}^{m_0} } y_{k_l}
        \in C_0^{\infty}(\omega)
        ~\text{where}~
        \mathcal J_{m}^{m_0}:=\{m,\cdots,m+m_0-1\}.
    \end{align}
    It follows from \eqref{zm-Jm-Defs-202202}, \eqref{yk-L2NormOne-WeakLimZero}, and \eqref{ykl-almost-OrthogonalAndConcentated-2202}
    that
    \begin{align}\label{zm-NormBounds-202202}
        \|z_m\|_{L^2(\Omega)}^2 =&  \sum_{l=m}^{m+m_0-1}
        \Big( \langle y_{k_l}, y_{k_l} \rangle_{L^2(\Omega)}
        +  2\sum_{m\leq l'<l}  \langle y_{k_l}, y_{k_{l'}} \rangle_{L^2(\Omega)}  \Big)
        \nonumber\\
        \in& \Big( m_0 - 2^{2-m},  m_0 +  2^{2-m} \Big).
    \end{align}
    Meanwhile, we can directly check that  for each $t>0$,
    \begin{align*}
        \chi_{(t_{m+m_0},t_m)}(t)z_m
        =& \bigg( \sum_{l'\in \mathcal J_{m}^{m_0}} \chi_{(t_{l'+1},t_{l'})}(t) \bigg)
        \sum_{l\in \mathcal J_{m}^{m_0}} y_{k_l}
        \nonumber\\
        =& \sum_{l\in \mathcal J_{m}^{m_0}} \chi_{(t_{l+1},t_l)}(t) y_{k_l}
        + \sum_{l,l'\in \mathcal J_{m}^{m_0}, l\neq l' } \chi_{(t_{l'+1},t_{l'})}(t) y_{k_l}.
    \end{align*}
    From the above, we can directly verify that
    \begin{align*}
        \int_0^{\varepsilon_0}  \|e^{tA} A^2 z_m\|_{L^2(\Omega)}  t^{\alpha} dt
        \geq&  \int_0^{\varepsilon_0} \|e^{tA} A^2 \chi_{(t_{m+m_0},t_m)}(t)z_m\|_{L^2(\Omega)}  t^{\alpha} dt
        \nonumber\\
        \geq& \int_0^{\varepsilon_0} \Big\|
        e^{tA} A^2 \sum_{l\in \mathcal J_{m}^{m_0}} \chi_{(t_{l+1},t_l)}(t) y_{k_l}
        \Big\|_{L^2(\Omega)}   t^{\alpha} dt
        \nonumber\\
        & -
        \sum_{l,l'\in \mathcal J_{m}^{m_0}, l\neq l' }
        \int_0^{\varepsilon_0}  \|e^{tA} A^2 \chi_{(t_{l'+1},t_{l'})}(t) y_{k_l}
        \|_{L^2(\Omega)}  t^{\alpha} dt,
    \end{align*}
    which, together with the definition of $\mathcal J_{m}^{m_0}$ (in \eqref{zm-Jm-Defs-202202}), implies that
    \begin{align*}
        \int_0^{\varepsilon_0} \|e^{tA} A^2 z_m\|_{L^2(\Omega)}  t^{\alpha} dt
        \geq& \sum_{l\in \mathcal J_{m}^{m_0}}
        \int_{t_{l+1}}^{ t_l}
        \| e^{tA} A^2 y_{k_l}\|_{L^2(\Omega)}  t^{\alpha} dt
        \nonumber\\
        & - \sum_{l\in \mathcal J_{m}^{m_0}} m_0
        \int_{\mathbb R^+ \setminus(t_{l+1}, t_l)}
        \| e^{tA} A^2 y_{k_l}\|_{L^2(\Omega)}  t^{\alpha} dt.
    \end{align*}
    The above, along with  \eqref{ykl-almost-OrthogonalAndConcentated-2202} and \eqref{upper-bound-Lp-heat-2202}, implies that
    \begin{align*}
        &\sum_{l\in \mathcal J_{m}^{m_0}}
        \int_{0}^{+\infty}
        \| e^{tA} A^2 y_{k_l}\|_{L^2(\Omega)}  t^{\alpha} dt
        \leq \sum_{l\in \mathcal J_{m}^{m_0}}
        \int^{t_{l+1}}_{ t_l}
        \| e^{tA} A^2 y_{k_l}\|_{L^2(\Omega)}  t^{\alpha} dt
        +   \sum_{l\in \mathcal J_{m}^{m_0}} 2^{- l  }
        \nonumber\\
        \leq&
        \bigg(  \int_0^{\varepsilon_0}
        \|e^{tA} A^2 z_m\|_{L^2(\Omega)}  t^{\alpha} dt
        + m_0\sum_{l\in \mathcal J_{m}^{m_0}} 2^{- l }
        \bigg)
        + \sum_{l\in \mathcal J_{m}^{m_0}} 2^{- l }
        \nonumber\\
        \leq& C_1 \| z_m\|_{L^2(\Omega)} + (m_0+1)\sum_{l\in \mathcal J_{m}^{m_0}} 2^{- l }.
    \end{align*}
    This, together with \eqref{zm-NormBounds-202202} and the definition of $\mathcal J_{m}^{m_0}$ (see \eqref{zm-Jm-Defs-202202}), leads to \eqref{etA-zm-Bounds-202202}.

    Finally, we will use \eqref{yk-L2NormOne-WeakLimZero} and \eqref{etA-zm-Bounds-202202}
    to show  $\alpha >1$.
    We observe that for each $z=\sum_{j\geq 1} a_j e_j \in C_0^\infty(\Omega)$,
    \begin{align*}
        &\int_0^{+\infty} \|e^{tA} A^2 z \|_{L^2(\Omega)} t^{ \alpha } dt
        =\int_0^{+\infty} \left\| \Big(
        e^{-t\eta_j } \eta_j^{2} a_j
        \Big)_{j\geq 1} \right\|_{\ell^{ 2 }} t^{\alpha}   dt
        \nonumber\\
        \geq&
        \sup_{\| (b_j)_{j\geq 1} \|_{ \ell^{ 2 } } \leq 1}
        \int_0^{+\infty}
        \bigg[
        \sum_{j\geq 1}
        \Big( e^{-t\eta_j } \eta_j^{2} a_j \Big) b_j
        \bigg]
        t^{\alpha} dt
        \nonumber\\
        =&
        \left\| \Big(
        \eta_j^{1-\alpha}  a_j
        \Big)_{j\geq 1} \right\|_{\ell^{ 2 }}
        \int_0^{+\infty} e^{-t} t^{\alpha} dt
        =  \left\| (-A)^{ 1-\alpha }  z \right\|_{L^2(\Omega)}
        \int_0^{+\infty} e^{-t} t^{\alpha} dt.
    \end{align*}
    From this and \eqref{etA-zm-Bounds-202202}, it follows that
    \begin{align*}
        \limsup_{m\rightarrow+\infty}
        \sum_{l=m}^{m+m_0-1}
        \left\| (-A)^{ 1-\alpha }  y_{k_l} \right\|_{L^2(\Omega)}
        \leq C_1 \sqrt{m_0 }
        \Big/\int_0^{+\infty} e^{-t} t^{\alpha} dt\;\;\mbox{for all}\;\; m_0\in \mathbb N^+\setminus\{1\}.
    \end{align*}
    Then, by \eqref{yk-L2NormOne-WeakLimZero}, after some direct computations, we determine that $\alpha>1$.
    This ends the proof of Proposition \ref{202202yb-prop-DoubleSidesOb-NecessaryWeights}.
\end{proof}

\subsection{Variation of constant formula}
This subsection provides a variation of the constant formula for equation \eqref{our-system}. We did not find this in the literature and present it here for the sake of completeness.

\begin{proposition}\label{prop-constant-variation}
	When $y_0\in L^2(\Omega)$ and
	 $u\in L^1_{loc}([0,+\infty); L^2(\Omega))$,
	\begin{align}\label{constan-variation-formula}
	y(t;y_0,u) = \varPhi(t) y_0  + \int_0^t \varPhi(t-s) (\chi_Q u)(s) ds,
	\;\; t\geq 0.
	\end{align}
\end{proposition}
\begin{proof}
	We arbitrarily fix  $y_0\in L^2(\Omega)$ and
$u\in L^1_{loc}([0,+\infty);L^2(\Omega))$. We simply write
	$y(\cdot)$ for the solution $y(\cdot;y_0,u)$.
	First, (\ref{constan-variation-formula}) is clearly true for $t=0$. We now fix  $t>0$
	and $z\in \mathcal H^2$. We write
	\begin{align}\label{4.11new,7.24}
	\varphi(s;z):=\varPhi(t-s)z,\;\;s\in [0,t].
	\end{align}
	Then, by \eqref{4.11new,7.24}, \eqref{varPhi-y-y0},
	and \eqref{our-obserble-system}, we see that $\varphi(\cdot;z)$ satisfies
	\begin{align}\label{varphi-z-equation}
	\varphi^\prime(s;z) + A\varphi(s;z) +  \int_s^t M(\tau-s) \varphi(\tau;z) d\tau=0,
	~~ s\in (0,t);\;\;\;\varphi(t;z)=z.
	\end{align}
	By \eqref{our-system} and (\ref{varphi-z-equation}), we find that
	\begin{align}\label{4.12,7.24}
	& \big\langle y(t),\varphi(t;z) \big\rangle_{L^2(\Omega)}
	-  \big\langle y_0, \varphi(0;z) \big\rangle_{L^2(\Omega)}
\nonumber\\
	=& \int_0^t \frac{d}{ds} \big\langle y(s;y_0,u), \varphi(s;z) \big\rangle_{L^2(\Omega)}   ds
	\\
	=& \int_0^t \Big\langle Ay(s) + \int_0^s M(s-\tau) y(\tau) d\tau + \chi_Q u(s) ,  \varphi(s;z) \Big\rangle_{ \mathcal H^{-2}, \mathcal H^2 } ds
	+  \int_0^t  \big\langle y(s), \varphi^\prime(s;z) \big\rangle_{L^2(\Omega)}  ds.
\nonumber
	\end{align}
	Meanwhile, by the Fubini theorem, it follows that
	\begin{align*}
	\int_0^t \bigg\langle \int_0^s M(s-\tau) y(\tau) d\tau, \varphi(s;z)
	\bigg\rangle_{ \mathcal H^{-2}, \mathcal H^2 } ds
	=& \int_0^t \int_\tau^t  M(s-\tau) \big\langle y(\tau),  \varphi(s;z) \big\rangle_{ \mathcal H^{-2}, \mathcal H^2 } ds d\tau
	\nonumber\\
	=& \int_0^t \bigg\langle y(s), \int_s^t M(\tau-s)  \varphi(\tau;z) d\tau \bigg\rangle_{L^2(\Omega)} ds.
	\end{align*}
	This, along with (\ref{4.12,7.24}), yields
	\begin{align*}
	&\big\langle y(t),\varphi(t;z) \big\rangle_{L^2(\Omega)}
	-  \big\langle y_0, \varphi(0;z) \big\rangle_{L^2(\Omega)}
\nonumber\\
	=&\int_0^t \Big\langle y(s)  , \varphi^\prime(s;z) + A\varphi(s;z) +  \int_s^t M(\tau-s) \varphi(\tau;z) d\tau \Big\rangle_{L^2(\Omega)} ds
		  + \int_0^t \langle  \chi_Q u(s) ,  \varphi(s;z) \rangle_{L^2(\Omega)}
	ds.
	\end{align*}
	The above, together with  \eqref{varphi-z-equation} and \eqref{4.11new,7.24}, indicates that
	\begin{align*}
	\big\langle y(t),z \big\rangle_{L^2(\Omega)} -  \big\langle y_0, \varPhi(t)z \big\rangle_{L^2(\Omega)}
	= \int_0^t \langle  \chi_Q u(s) ,  \varPhi(t-s)z \rangle_{L^2(\Omega)} ds.
	\end{align*}
	Since $z$ was arbitrarily selected from $\mathcal H^2$, we can use a standard density
	argument in the above equality,
	as well as the first equality in \eqref{eq-varPhi-expression},
	to obtain (\ref{constan-variation-formula}). This completes the proof.
	\end{proof}

\subsection{Functional framework}

The following lemma is cited from \cite[Lemma 5.1]{WWZ-jems}.

\begin{lemma}\label{lemma-0428-fn}
Let $\mathbb K$ be either $\mathbb R$ or $\mathbb C$, and let $X$, $Y$, and $Z$ be three Banach spaces over  $\mathbb K$, with their dual spaces $X^*$, $Y^*$ and $Z^*$, respectively. Let $R\in \mathcal L(Z,X)$ and $O\in \mathcal L(Z,Y)$. Then, the following two propositions are equivalent:

\noindent (i) There exists a $\widehat C_0>0$ and an $\hat\varepsilon_0>0$ such that
\begin{eqnarray}\label{lemma-202202-fn-ii}
 \| R z \|^2_X  \leq   \widehat C_0 \|Oz\|^2_Y
  +  \hat\varepsilon_0 \|z\|_Z^2\;\;\mbox{for all}\;\;z\in Z.
\end{eqnarray}

\noindent (ii) There exists  a $C_0>0$ and an $\varepsilon_0>0$ such that for each $x^*\in X^*$, there is a
$y^*\in Y^*$ that satisfies
\begin{eqnarray}\label{lemma-202202-fn-i}
 \frac{1}{C_0} \|y^*\|^2_{Y^*} + \frac{1}{\varepsilon_0} \|R^*x^*-O^*y^*\|^2_{Z^*}
 \leq \|x^*\|^2_{X^*}.
\end{eqnarray}
Furthermore, when one of the above two propositions holds, the constant pairs $(C_0,\varepsilon_0)$ and $(\widehat C_0,\hat\varepsilon_0)$ can be chosen to be the same.
\end{lemma}

The following result is a consequence of Lemma \ref{lemma-0428-fn}.

\begin{corollary}\label{202202yb-corollary-framework}
With the notation   in Lemma \ref{lemma-0428-fn},  the following two propositions are equivalent:

\noindent (i) There exists a $C_1>0$ such that
\begin{eqnarray*}
 \| R z \|_X  \leq   C_1 \|Oz\|_Y\;\;\mbox{for all}\;\;z\in Z.
\end{eqnarray*}

\noindent (ii) There exists  a $C_2>0$  such that for each $x^*\in X^*$, there is a
$y^*\in Y^*$ that satisfies
\begin{eqnarray*}
R^*x^*-O^*y^*=0 ~\text{in}~ Z^*
~\text{and}~
 \|y^*\|_{Y^*}  \leq C_2 \|x^*\|_{X^*}.
\end{eqnarray*}

Furthermore, when one of the above two propositions holds, the constants $C_1$ and $C_2$ can be chosen to be the same.
\end{corollary}

\begin{proof}
We first show $(i)\Rightarrow(ii)$. Suppose that $(i)$  holds (with $C_1>0$). Then, for each $\hat \varepsilon_0>0$, we have \eqref{lemma-202202-fn-ii} (with $\widehat C_0=C_1^2$). This, along with
Lemma \ref{lemma-0428-fn}, leads to \eqref{lemma-202202-fn-i}, where  $y^*$ and
$(C_0,\varepsilon_0)$ are replaced by
$y^*_{\hat\varepsilon_0}$ and $(C_1^2,\hat\varepsilon_0)$, respectively.
Hence, the family  $\{y^*_{\hat\varepsilon_0} \}_{\hat \varepsilon_0>0}$ is bounded in $Y^*$. Thus, there is a subsequence $\{y^*_{\varepsilon_k} \}_{k\geq 1}$ that converges to $\hat y^*$ weakly star  in $Y^*$. Thus, $(ii)$   holds for $y^*=\hat y^*$ and $C_2=C_1$.

We next show $(ii)\Rightarrow(i)$. We suppose that $(ii)$ holds (with $C_2>0$). Then, for each
 $ \varepsilon_0>0$, we have \eqref{lemma-202202-fn-i}, with $C_0=C_2^2$. This, together with Lemma \ref{lemma-0428-fn},  yields  \eqref{lemma-202202-fn-ii} where
 $(\widehat C_0,\hat\varepsilon_0)$ is replaced by $(C_2^2,\varepsilon_0)$. Hence,  $(i)$
  holds for  $C_1=C_2$. This completes the proof.
\end{proof}

\bigskip
%\footnotesize
\noindent\textit{Acknowledgments.}
The first two authors would like to gratefully thank Dr. Huaiqiang Yu for his valuable comments and suggestions.

The first two authors were partially supported by the National Natural Science Foundation of China under grants   12371450 and 12171359, respectively. The second author was also founded by the Humboldt Research Fellowship program from the Alexander von Humboldt Foundation. 

The  third author has been funded by the Alexander von Humboldt-Professorship program, the Transregio 154 Project “Mathematical Modelling, Simulation and Optimization Using the Example of Gas Networks" of the DFG, the ModConFlex Marie Curie Action, HORIZON-MSCA-2021-DN-01, the COST Action MAT-DYN-NET, grants PID2020-112617GB-C22 and TED2021-131390B-I00 of MINECO (Spain), and by the Madrid Goverment -- UAM Agreement for the Excellence of the University Research Staff in the context of the V PRICIT (Regional Programme of Research and Technological Innovation).

%\bigskip\bigskip
%\noindent\textbf{Data Availability Statement} Data sharing is not applicable to this paper because no
%    dataset was analysed or generated  during the study.
%    
%\bigskip
%\noindent\textbf{Statements and Declarations}
%
%\noindent\textbf{Competing Interests} The authors declare that  they have no financial interests nor any conflict of interest.

\end{document}